\documentclass[11pt]{amsart}
\usepackage[T1]{fontenc}

\usepackage{lmodern, amsfonts,amsmath,amstext,amsbsy,amssymb,
amsopn,amsthm,upref,eucal,mathptmx}

\usepackage{mathrsfs}

\RequirePackage{xcolor} 
\definecolor{halfgray}
{gray}{0.55}
\definecolor{webgreen}
{rgb}{0,0.4,0}
\definecolor{webbrown}
{rgb}{.8,0.1,0.1}
\definecolor{red}
{rgb}{1,0,0}
\usepackage{microtype}

\newcommand \R {{ \mathbb R}}
\def\C{{\mathbb C}}
\newcommand \Z {{ \mathbb Z}}
\newcommand \N {{ \mathbb N}}
\newcommand \T {{ \mathbb T}}
\newcommand \Q {{ \mathbb Q}}

\newtheorem{theorem}{Theorem}[section]
\newtheorem {lemma} [theorem]{Lemma}

\newtheorem{corollary}[theorem]{Corollary}
\newtheorem{remark}[theorem]{Remark}

\newtheorem{definition}[theorem]{Definition}

\title[Time-changes of Heisenberg nilflows ]%
{ Time-changes of Heisenberg nilflows }

  \author{Giovanni Forni}
\author{Adam Kanigowski}

\address{Department  of Mathematics\\
  University of Maryland \\
  College Park, MD USA}
  
  \address{Department  of Mathematics\\
  Pennsylvania State University \\
  State College, PA USA}

\email{gforni@math.umd.edu}

\email{adkanigowski@gmail.com}
\subjclass
        {    }

\date{\today}
\dedicatory{\text{\rm In memory of Jean-Christophe Yoccoz}, \\ il miglior fabbro. }

\begin{document}

\begin{abstract}
  We consider the three dimensional Heisenberg nilflows. Under a full measure set Diophantine condition on the generator of the flow we construct {\it Bufetov functionals} which are asymptotic to ergodic integrals for sufficiently smooth functions, have a modular property and  scale exactly under the renormalization dynamics. By the asymptotic 
  property we derive results on limit distributions, which generalize earlier work of Griffin and Marklof \cite{GM} and Cellarosi and Marklof~\cite{CM}. We then prove analyticity of the functionals in the transverse directions  to the flow. As a consequence of this analyticity property we derive that there exists a full measure set of nilflows such that {\it generic} (non-trivial) time-changes are mixing and moreover have a ``stretched polynomial'' decay of correlations for sufficiently smooth functions (this strengthens a result of Avila, Forni,  and Ulcigrai~\cite{AFU}). Moreover we also prove that there exists a full Hausdorff dimension set of nilflows such that generic non-trivial time-changes have polynomial decay of correlations. 
\end{abstract}
\def\echo#1{\relax}
\maketitle

\section{Introduction}

This paper concerns  the smooth ergodic theory of parabolic flows, that is, flows characterized by polynomial (sub-exponential) divergence of nearby orbits.
 In particular we prove results  on limit distributions of Heisenberg nilflows and on the decay of correlations  of their non-trivial reparametrizations (time-changes). 
 Our approach is based on the construction of  {\it finitely additive H\"older measures} and {\it H\"older cocycles} for Heisenberg nilflows, asymptotic to ergodic integrals, following 
 the work of A.~Bufetov \cite{Bu} on translation flows and of Bufetov and G.~Forni~\cite{BF}, \cite{Fo1} on horocycle flows. H\"older cocycles for translation flows are closely related to ``limit shapes'' of ergodic sums for Interval Exchange Transformations, studied in the work of  S.~Marmi, P.~Moussa and J.-C.~Yoccoz \cite{MMY} on wandering intervals for affine Interval Exchange Transformations. In fact, roughly speaking, ``limit shapes'' are related to graphs of H\"older cocycles as functions of time.

We recall that the mixing property for generic, non-trivial time-changes of Heisenberg nilflows was proved by A.~Avila, G.~Forni and C.~Ulcigrai~\cite{AFU}. The main result of that paper was that for uniquely ergodic Heisenberg nilflow all non-trivial time-changes, within a dense subspace of time-changes, are mixing.  Under a Diophantine condition  the set of trivial time-changes has countable codimension and can be explicitly described in terms of invariant distributions for the nilflow.  

Results on limit theorems  for skew-translations, which appear as return maps (with constant return time) of Heisenberg nilflows, limited however  to a single character function, have more recently been proved by J.~Griffin and J.~Marklof~\cite{GM} and refined by F.~Cellarosi and Marklof~\cite{CM}  by an approach based on theta functions.  Their worked raised the question of possible relations between theta functions and Bufetov's H\"older cocycles, developed for other analogous dynamical systems in \cite{Bu} (translation flows), \cite{BF} (horocycle flows), \cite{BS} (tilings),  as a formalism to derive asymptotic theorem for ergodic averages and prove limit theorems. 

\smallskip

In this paper we generalize the results of Griffin and Marklof \cite{GM} on limit distributions, proving in particular that almost all limits of ergodic averages of arbitrary sufficiently smooth functions are distributions of  H\"older continuous functions on the Heisenberg nilmanifold, hence in particular they they have compact support.  Our main results, however, are on the decay of correlations of smooth functions for time-changes: we prove that it has polynomial (power law) speed for all non-trivial smooth time-changes of Heisenberg nilflows of bounded type, within a generic subspace of time-changes. As mentioned above, the study of limit distributions for parabolic flows has been developed only in recent years after Bufetov's work \cite{Bu} on translations flows
(and Interval Exchange Transformations). A comprehensive study of spatial and temporal limit theorems for dynamical systems of different type has been more recently carried
out in the work of D.~Dolgopyat and O.~Sarig \cite{DS}. 

\smallskip
The study of mixing properties of elliptic parabolic flows and their time-changes has a longer history.  For
instance, mixing properties of suspension flows over rotations and Interval Exchange Transformations have been investigated in depth (see for instance \cite{Ko1}, \cite{Ko2},
\cite{KS}, \cite{Sch},\cite{Ul1}, \cite{Ul2} and reference therein), mixing for reparametrizations of linear toral flows were investigated by B.~Fayad (see for instance
\cite{Fa2}), finally mixing for time-changes of classical horocycle flows was proved in a classical paper of B.~Marcus~\cite{Ma} after a partial result of Kushnirenko
\cite{Ku}.  As mentioned above mixing for time-changes of Heisenberg nilflows was investigated in \cite{AFU}. Work in progress of Avila, Forni, Ravotti and Ulcigrai indicates
that the methods developed there extend to proof of mixing for a dense set of nontrivial time-change for any uniquely ergodic nilflows.  Ravotti's paper~\cite{Rav2} is a step
in that direction. It should be remarked that there is an important difference between time-changes of linear toral flows and parabolic flows. In the parabolic case there
are often countably many obstructions to triviality of time-changes for Diophantine flows, while in the elliptic case of linear toral flows non-trivial time-changes can exist
only in the Liouvillean case.

\smallskip
Estimates on the decay of correlations of smooth functions for non-homogenous elliptic or parabolic flows are harder to come by and there are much fewer results in the literature. A classical paper of M.~Ratner established the decay rate for classical horocycle flows (as well as geodesic flows) on surfaces of constant negative curvature.
This result was generalized to sufficiently smooth time-changes of horocycle flows by Forni and Ulcigrai~\cite{FU}, who also proved that the spectrum remains Lebesgue. 
Fayad~\cite{Fa1} proved polynomial decay for a class of Kochergin-type flows on the $2$-torus and only recently, in \cite{FFK},  it was shown that there exists a class of Kochergin flows on the $2$-torus with Lebesgue maximal spectral type.  For locally Hamiltonian flows with a saddle loop on surfaces (or, more generally, for suspension flows over Interval Exchange Transformations with asymmetric logarithmic singularities of the roof function), Ravotti \cite{Rav1} was able to  prove (logarithmic) estimates on decay of correlations. For these flows mixing was proved by Khanin and Sinai~\cite{KS} in the toral case, and by C.~Ulcigrai \cite{Ul1} for suspension flows over 
Interval Echange Transformations in the significant special case of roof functions with a single asymmetric logarithmic singularity. 

We expect non-trivial time-changes of nilflows to have polynomial decay of correlations. However, we are able to prove this result only for Heisenberg nilflows of bounded
type. Our methods do not generalize to higher step nilflows, since they are based on the renormalization dynamics introduced by L.~Flaminio and G.~Forni in \cite{FlaFo},
which has no known generalization to the higher step case. We are also unable to decide whether the spectral measures of time-changes of Heisenberg nilflows are absolutely
continuous with respect to Lebesgue. Indeed, the approach of \cite{FU}, considerably refined in  \cite{FFK} , fails since the ``stretching of Birkhoff sums'' is at best borderline
square integrable (it grows at most as the square root of the time, up to logarithmic terms). In fact, our bounds on the decay of correlations are significantly worse than that,
and we have no control on the size of the exponent. This follows from the general principle that proving ``lower bounds'' on Birkhoff sums or ergodic integrals is much harder than proving ''upper bounds''. In our case we are able to prove polynomial (power-law) lower bounds outside appropriate sublevel sets of Bufetov's H\"older cocycles, which are asymptotic to ergodic integrals up to a well-controlled error. Polynomial estimates on the measure of such sublevel sets (for small parameter values) are derived from general results (see \cite{Bru}, \cite{BruGa}) on the measure of the sublevel sets of {\it analytic functions}. In fact,  at the core of our argument we establish the real analyticity of the Bufetov cocycles along the leaves of a foliation transverse to the flow. 

This outline is different from the proof of mixing in \cite{AFU}. In that paper  the stretching of Birkhoff sums for Heisenberg nilflows was derived from a more general  result on the growth of Birkhoff sums of functions which are not coboundaries with \emph{measurable }transfer function, essentially based on a measurable Gottschalk-Hedlund theorem, and on the parabolic divergence of orbits. However, it is completely unclear whether it is possible to prove an effective version of this argument. For this reason we have followed
here a different approach.

\medskip

\paragraph{\it Outline of the paper.} In Section \ref{sec:def} we give basic definitions on Heisenberg nilflows, the Heisenberg moduli space, renormalization flow and Sobolev spaces. Finally we state two main theorems. In Section \ref{sec:rep} we recall some basic results in representation theory of Heisenberg group. In Section \ref{sec:stretch} we compute the stretching of  arcs (in the central direction) under the reparametrized flow. Sections~\ref{sec:buf} and~\ref{sec:Main_properties}  are crucial since {\it Bufetov functionals} are constructed and their main properties are studied. In particular we prove the expected asymptotic formula according to which Bufetov fuctionals control orbital integrals. In Section~\ref{sec:limit_dist} we derive from the asymptotic formula results on limit distributions of ergodic integrals for Diophantine Heisenberg niflows, following the method developed in \cite{Bu}, \cite{BF}. We also give an alternative proof, based on representation theory, of a substantial part of the work of Griffin and Marklof~\cite{GM} 
on limit theorems for skew-shifts of the $2$-torus, and generalize most of their conclusions to arbitrary smooth functions.  Our approach also naturally gives results on the regularity of limit distributions, in particular their H\"older property (with exponent $1/2-$) first derived for quadratic Weyl sums  in the work of Cellarosi and Marklof \cite{CM}. 

In Section~\ref{sec:square_mean_bounds} we prove sharp square mean lower bounds for Bufetov functionals along the leaves of a one-dimensional foliation transverse to the flow. Our aim is to prove measure estimates for the sub-level sets of Bufetov functionals, a key result in establishing the stretching of ergodic integrals outside sets of small measure. For that we prove in Section~\ref{sec:analyticity} that Bufetov functionals are {\it real analytic} on the leaves of a $2$-dimensional foliation (the weak-stable foliation of the renormalization dynamics on the Heisenberg nilmanifold). We then recall in Section~\ref{sec:valency} a result of A.~Brudnyi~\cite{Bru} on the measure of the sub-level sets of real analytic functions. These estimates depend on the so-called {\it Chebyshev degree} and {\it valency} of the function. We prove that under certain conditions the valency is uniformly bounded on every normal family of analytic functions. 
In Sections~\ref{sec.mes} and~\ref{sec.mesgen}  we apply results of the previous section to finally prove measure estimates on the sets where Bufetov functionals are small (Lemmas~\ref{lemma:compact} and~\ref{lemma:gencase}). We conclude in  Section \ref{sec:cor} with an analysis of  correlations and derive from results of Sections \ref{sec.mes} and~\ref{sec.mesgen} (Corollary \ref{cor:compact} and \ref{cor:gencase}) the proof of our main Theorems \ref{mainthm1} and \ref{mainthm2}.

\section{Definitions}\label{sec:def}
In this section we will recall definitions of Heisenberg nilflows, moduli space of Heisenberg frames $\mathcal{M}$, the renormalization flow $g_\R$ on $\mathcal{M}$ and
te renormalization cocycle $\rho_\R$ on the Hilbert bundle of Sobolev distributions. For more details see \cite{FlaFo} or \cite{Fo2}. We also introduce an extended renormalization flow $\hat g_\R$ on an extended moduli space $\hat{\mathcal{M}}$, which is a tautogical bundle over $\mathcal{M}$ with fibers isomorphic to the Heisenberg nilmanifold.

\subsection{Nilflows}
The (three-dimensional) Heisenberg group $\text{\rm Heis}$ is given by 
$$
\text{\rm Heis}:=\left\{\begin{pmatrix}1&x&z\\0&1&y\\0&0&1\end{pmatrix}\;:\; x,y,z\in \R\right\}.
$$
Let $\Gamma$ be a lattice in $\text{\rm Heis}$. The {\em Heisenberg manifold} $M$ is the quotient $\Gamma \backslash \text{\rm Heis}$. It is known that up to an automorphism of $\text{\rm Heis}$
$$
\Gamma=\Gamma_K=\left\{\begin{pmatrix}1&m&\frac{p}{K}\\0&1&n\\0&0&1
\end{pmatrix}\;:\; m,n,p\in \Z\right\},
$$
where $K$ is a positive integer. Notice that $M$ has a probability measure {\rm vol} locally given by Haar measure on $\text{\rm Heis}$.\\
\subsubsection{Heisenberg nilflows}
Let $W$ be any element of the Lie algebra $\eta$ of $\text{\rm Heis}$. The {\em Heisenberg nilflow} for $W$ is given by
$$
\phi^W_t(x)=x\exp(tW) \text{ for } x\in M.
$$
Notice that $\phi^W_t$ on $M$ preserves $\rm vol$.

\subsection{Renormalization}
A {\em Heisenberg frame} is any triple $(X,Y,Z)$ of elements generating $\eta$ such that $Z$ is a fixed generator of the center of the Lie algebra $Z(\eta)$ and $[X,Y]=Z$ (of course we have $[X,Z]=[Y,Z]=0$). One can for instance take 
$$Z=Z_0=\begin{pmatrix}0&0&1\\0&0&0\\0&0&0\end{pmatrix}.$$
The set of all Heisenberg frames can be identified with the subgroup $A$ of all automorphisms of $\text{\rm Heis}$ which are identity on the center. Notice that up to identification, $A$ is equal to the group $SL(2,\R)\ltimes \R^{*^2}$. Let $A_\Gamma$ be the subgroup of $A$ which stabilizes $\Gamma$, i.e. for $a\in A_\Gamma$, $a(\Gamma)=\Gamma$. We have the following definition
\begin{definition}\label{modsp}\cite{FlaFo} The {\em moduli space} of the Heisenberg manifold $M$ is the quotient space $\mathcal{M}=A_\Gamma \backslash A$.
\end{definition}
It follows that $A_\Gamma$ is isomorphic to $\Lambda_\Gamma\ltimes (K^{-1}\Z^*)^2$ where $\Lambda_\Gamma$ is a finite index subgroup of $SL(2,\Z)$. Therefore the space $\mathcal{M}$ is a finite volume orbifold which fibers over the homogeneous space $\Lambda_\Gamma\backslash SL(2,\R)$ with fiber $\T^2$ (see Proposition 3.4. in \cite{FlaFo}). 

 \subsubsection{The renormalization flow}
Following the notation from \cite{FlaFo}, for an element $a\in A$ we denote $\bar{a}:=A_\Gamma a\in \mathcal{M}$. Let $(X_0,Y_0,Z_0)$ be a fixed Heisenberg triple.  Let $(a_t)$ be the following one-parameter subgroup of $A$:
$$
a_t(X_0,Y_0,Z_0)=(e^tX_0,e^{-t}Y_0,Z_0).
$$
\begin{definition}\label{renflo}The {\em renormalization flow} $g_\R$ on $\mathcal{M}$ is defined by 
$$
g_t(\bar{a})=\bar{a}a_t=A_\Gamma aa_t.
$$
\end{definition}
In what follows we will also consider the {\em extended renormalization flow} ${\hat g}_\R$ on  {\em extended moduli space} $\hat {\mathcal{M}}$, defined as follows.
The extended moduli space is the quotient
$$
 \hat {\mathcal{M}} :=  A_\Gamma \backslash (A\times M) \,,
$$
with respect to the action of $A_\Gamma$ on $A\times M$ by multiplication on the left on $A$ and by the embedding $A_\Gamma <\text{Diff}(M)$ on $M$. The 
extended renormalization flow is the projection to the extended moduli space  of the flow
$$
(t, a,x)   \to   (aa_t, x)  \,, \quad \text{ for all } (t,a,x) \in \R\times A\times M\,.
$$
Note that $ \hat {\mathcal{M}} $ is a fiber bundle over ${\mathcal{M}}$ with fiber diffeomorphic to $M$ and the extended renormalization flow
$\hat g_\R$ projects onto the renormalization flow $g_\R$ .

\subsection{Sobolev spaces}
For $(a,x)\in A\times M$ denote $X_a=a_\ast(X_0)$, $Y_a=a_\ast(Y_0)$. Let $\Delta_a:=-X_a^2-Y_a^2-Z_0^2$ be the Laplace operator. For every $s\in \R$ and any $C^{\infty}$ function $f\in L^2(M)$ we define
$$
\|f\|_{a,s}=\langle f, (1+\Delta_a)^sf\rangle ^{1/2}.
$$
Let $W^s_a(M)$ be the completion of $C^{\infty}(M)$ with the above norm. Let $W^{-s}_a(M)$ denote the dual space. Following \cite{FlaFo} again, we can define 
$$
\textbf{W}^s:=A_\Gamma\backslash (A\times W^s(M))\text{ and } \textbf{W}^{-s}:=A_\Gamma\backslash (A\times W^{-s}(M)),
$$
where $A\times W^s(M)$ ($A\times W^{-s}(M)$) denotes the Hilbert bundle over $A$, where 
$$
\|(a,f)\|_s=\|f\|_{a,s}\;\; (\|(a,D)\|_{-s}=\|D\|_{a,-s}).
$$
We denote elements of $\textbf{W}^s$ (respectively $\textbf{W}^{-s}$) by 
$\overline{(a,f)}$ (respectively $\overline{(a,D)}$).

We now define the {\em renormalization cocycle}.
\begin{definition}\cite{FlaFo} 
\label{def:currents_cocycle}
The renormalization cocycle $\rho_\R$ is a flow on $\textbf{W}^s$ and $\textbf{W}^{-s}$ given by 
$$
{\rho_t}\overline{(a,f)}=\overline{(aa^t,f)}\quad \text{  and  } \quad 
{\rho_t}\overline{(a,D)}=\overline{(aa^t,D)}.
$$
\end{definition}
\textbf{Main results.}

Let $X$ be the generator of $(\phi_t^X)$ and $\omega_X$ denote the measure preserved by $X$.
\begin{theorem}\label{mainthm1} There exists a set $\mathcal{F}\subset \mathcal{M}$ of full Hausdorff dimension and, for all $s>7/2$, a generic set $\Omega\subset W^s(M)$  such that, for $a=(X,Y,Z)\in \mathcal{F}$ and $\alpha\in W^s_+(M)$ with $Z(1/\alpha)\in \Omega$ the following holds. Either $1/\alpha$ is an $X$-coboundary, or there exist constants $C_{a,\alpha}>0$ and $\delta_{a,\alpha}>0$ such that, for  every $h\in W^s(M)$, $g\in L^2(M)$ such that $Zg\in L^2(M)$ and for all $t \in \R$, we have
$$
|< h\circ \phi^{\alpha X}_t , g>_{L^2(M,\omega_{\alpha X})}|<\frac{C_{a,\alpha}} {(1+\vert t \vert) ^{\delta_{a,\alpha}}     } \Vert h \Vert_s (\Vert g \Vert_0 + \Vert Zg\Vert_0)  \,.
$$
\end{theorem}

\begin{theorem}\label{mainthm2}
There exists a set $\mathcal{F}'\subset \mathcal{M}$ of full measure and, for all $s>7/2$, a generic set $\Omega\subset W^s(M)$ such that for $a=(X,Y,Z)\in \mathcal{F}'$ 
and $\alpha\in W^s_+(M)$ with $Z (1/\alpha)\in \Omega$ the following holds. Either $1/\alpha$ is an $X$-coboundary, or for every $\delta>1/2$ there exists a constant 
$C_{a,\alpha,\delta}>0$  such that, for  every $h\in W^s(M)$, $g\in L^2(M)$ such that $Zg\in L^2(M)$ and for all $t \in \R$, we have
$$
|< h\circ \phi^{\alpha X}_t , g>_{L^2(M,\omega_{\alpha X})}|< C_{a,\alpha, \delta} (1+ \vert t\vert) ^{ - \frac{1}{1+\log^\delta (1+\vert t \vert)  }  }
\Vert h \Vert_s (\Vert g \Vert_0 + \Vert Zg\Vert_0) \,.
$$
\end{theorem}

\section{Representation theory}\label{sec:rep}
Recall that the right quasi regular representation $\mathcal{U}$ of the Heisenberg group $\text{\rm Heis}$ on $L^2(M,\mu)$ is given by 
$$ 
\mathcal{U}(g)F=F(R(g)),
$$
here $R(g)(x)=xg$. Notice that 
$$
L^2(M,\mu)=\bigoplus_{n\in\Z}H_n,
$$
where $H_n=\{f\in L^2(M,\mu)\;:\; \exp(tZ)f=\exp(2\pi \imath n Kt)f\}$ are closed and $\mathcal{U}$-invariant. Moreover each $H_n$ further splits into irreducible subrepresentations spaces. A complete classification of irreducible representations (with non-zero central parameter) is given by the Stone-Von Neumann theorem. 
\begin{theorem}[Stone-Von Neumann]  For any a Heisenberg triple $a=(X,Y,Z)$ and for any irreducible unitary representation $\pi$ of $\text{\rm Heis}$ of non-zero central parameter $n\in \Z\setminus \{0\}$ on the Hilbert space $H\subset H_n$ there exists a unique unitary operator $U:= U^H_a: H \to L^2(\R, \lambda)$ such that 
$$
\begin{aligned}
 &(U \circ D\pi (X) \circ U^{-1})(f) (u)  = f'(u) \,, \\  &(U\circ  D\pi (Y)\circ U^{-1})(f) (u)  = 2\pi \imath n K u f(u)\,, \\   &(U\circ D\pi (Z) \circ U^{-1})(f) (u) =  2\pi \imath nK  f(u)\,.
\end{aligned}
$$
\end{theorem}

Moreover, by Proposition 4.4. in \cite{FlaFo} it follows that for every irreducible representation $H\subset H_n$ with non-zero central parameter the space of $X$-invariant distribution  has dimension $1$. 

\begin{definition}\label{normdist}
For all $a=(X,Y,Z)\in A$, let $D^H_a$ be the unique distribution such that $D^H_a$ corresponds by the unitary equivalence $U^H_a$ (given by the Stone-Von Neumann theorem)  to the Lebesgue measure on $\R$. 
\end{definition}
\begin{lemma}\label{scaldist}If $H$ is any irreducible representation of non-zero central parameter, we have 
$$
D^H_{g_t(a)}=e^{-t/2}D^H_{a}.
$$
\end{lemma}
\begin{proof} Let $H$ be any irreducible unitary representation of central parameter $n\neq 0$. The unitary operator $U_t:L^2(\R,\lambda)\to L^2(\R,\lambda)$ given by 
\begin{equation}
\label{eq:intop}U_t(f)=e^{t/2}f(e^tu)
\end{equation}
 intertwines $(X,Y,Z)$ and $g_t(X,Y,Z)=(e^tX,e^{-t}Y,Z)$, in the sense that
 \begin{equation}
 \label{eq:inteq}
\begin{aligned}
 &U_t(U \circ D\pi (e^tX) \circ U^{-1})U_t^{-1} = U \circ D\pi (X) \circ U^{-1} \,, \\  &U_t(U\circ  D\pi (e^{-t}Y)\circ U^{-1})U_t^{-1}  = U\circ  D\pi (Y)\circ U^{-1}\,, \\   
 &U_t(U\circ D\pi (Z) \circ U^{-1})U_t^{-1}
 = U\circ D\pi (Z) \circ U^{-1}\,.
\end{aligned}
\end{equation}
 It follows by the above definitions and by the uniqueness part of the Stone-Von Neumann theorem that $U^H_{g_t(X,Y,Z)}= U_t \circ U^H_{X,Y,Z}$,
 hence by the definition of $D^H_{X,Y,Z}$ it follows that 
$$
\begin{aligned}
D^H_{g_t(a)}&= \text{Leb}_\R \circ U^H_{g_t(a)} =  (\text{Leb}_\R \circ U_t) \circ  U^H_a \\ &=  e^{-t/2} (\text{Leb}_\R \circ  U^H_a)  =   e^{-t/2} D^H_a.
\end{aligned}
$$
This finishes the proof. 
\end{proof}

\section{Stretching of curves}\label{sec:stretch}
Fix a Heisenberg triple $(X,Y,Z)$. Let $\alpha>0$ denote a smooth time-change function (of the flow $\phi^V_\R$ generated by $X$) and $V= \alpha X$. We have the commutations
$$
\begin{aligned}
&[V,Y] = [\alpha X,Y] = -(Y\alpha) X +  \alpha Z =   -\frac{Y\alpha}{\alpha} V  + \alpha Z\,,  \\  
&[V,Z] = [\alpha X,Z] =  -\frac{Z\alpha}{\alpha} V  \,.
\end{aligned}
$$
Let $(\phi^V_t)$ denote the flow generated by the vector field $V$ on the nilmanifold $M$. We will compute the tangent vector of the push forwards of curves under the flow 
$\phi^V_\R$. Let $W$ be any vector in the Lie algebra.  We write 
$$
(\phi^V_t)_* (W) = a_t V + b_t Y + c_t Z \,.
$$
By differentiation we derive
$$
\begin{aligned}
 \frac{da_t}{dt} V + \frac{db_t}{dt} Y + \frac{dc_t}{dt} Z &= -Va_t V  - Vb_t Y -  b_t [V,Y]  - Vc_t Z - c_t [V,Z]\\ 
 &=  -(Va_t   -b_t \frac{Y\alpha}{\alpha} - c_t\frac{Z\alpha}{\alpha} ) V  - Vb_t Y  - (b_t \alpha + Vc_t) Z \,.
 \end{aligned}
$$
or in other terms
$$
\begin{aligned}
 \frac{da_t}{dt} &=  -Va_t   + b_t \frac{Y\alpha}{\alpha} + c_t\frac{Z\alpha}{\alpha}  \,,\\
  \frac{db_t}{dt} &=  - Vb_t  \,, \\
\frac{dc_t}{dt} &=  - Vc_t - b_t \alpha \,.
\end{aligned}
$$
It follows that
$$
\begin{aligned}
 &\frac{d}{dt} (a_t \circ \phi^V_t) =   (b_t \circ \phi^V_t) \frac{Y\alpha}{\alpha}\circ \phi^V_t + (c_t\circ \phi^V_t) \frac{Z\alpha}{\alpha}\circ \phi^V_t  \,,\\
 &\frac{d}{dt}  (b_t \circ  \phi^V_t) =  0 \,, \\
&\frac{d}{dt} (c_t\circ \phi^V_t) =  -(b_t \circ \phi^V_t) (\alpha\circ \phi^V_t)  \,.
\end{aligned}
$$
At this point analogously to \cite{AFU} we will look at the case $W=Z$ (curves tangent to the central direction), hence $(a_0, b_0,c_0)= (0,0, 1)$. We have
$$
\begin{aligned}
 &a_t \circ \phi^V_t =  \int_0^t   \frac{Z\alpha}{\alpha}\circ \phi^V_\tau  d\tau   \,,\\
 &b_t \circ \phi^V_t =  0 \,, \\
& c_t\circ \phi^V_t =  1  \,.
\end{aligned}
$$
In other terms
$$
D \phi^V_t (Z) =   ( \int_0^t   \frac{Z\alpha}{\alpha}\circ \phi^V_\tau  d\tau) V  +  Z \,.
$$

To understand the above orbital integrals we  write them in terms of the nilflow $\phi^X_\R$. We have
relations
$$
\tau_V (x,t )  = \int_0 ^t  \alpha^{-1}  \circ \phi^X_r (x) dr \quad \text{ and } \quad 
\tau_X (x,t)  = \int_0 ^t  \alpha  \circ \phi^V_r (x) dr
$$
By these formulas and  by change of variables, we have
\begin{equation}\label{tmch}
\int_0^t   f\circ \phi^V_\tau (x) d \tau  =  \int_0^{\tau_X(x,t)}  (\frac{f}{\alpha})\circ \phi^X_r (x) dr 
\end{equation}

We will therefore investigate time averages
\begin{equation}\label{tmch2}
 \int_0^t  f \circ \phi^X_r (x) dr
\end{equation}
for functions $f$ of zero average with respect to the Haar volume on $M$.

\section{Construction of the functionals}\label{sec:buf}
Let $\gamma$ be any rectifiable curve.  The curve $\gamma$ defines a current, that is, a continuous functional on $1$-forms. 
We recall that the renormalization cocycle $(\rho_t)$ acts on  currents (see Definition~\ref{def:currents_cocycle}). 

Fix an irreducible representation $H\subset L^2_0(M)$ contained in the eigenspace of eigenvalue $2\pi \imath Kn\in 2\pi \imath K \Z$ of the action of the center
of the Heisenberg group on $M$ and fix a Heisenberg triple $a=(X,Y,Z)$.  There exists a unique {\it basic current} $B^H_{a}$ (of degree $2$ and dimension $1$)
associated to $D^H_{a}$. Let $\omega$ denote the invariant volume form (with unit total volume) and let $\eta_X= \imath_X \omega$. The
basic current $B^H_a$ is defined as 
 $$
B^H_a =   D^H_a \eta_X \,.
$$
The above formula means that for every $1$-form $\alpha$ we have 
$$
B^H_a (\alpha) = D^H_a ( \frac{\eta_X \wedge \alpha}{\omega})\,.
$$
The current $B^H_a$ is basic in the sense that
$$
\imath_X B^H_a = {\mathcal L}_X  B^H_a = 0\,.
$$
The basic current $B^H_a$ belongs to a dual Sobolev space of currents, defined as follows. We can write any smooth 
$1$-form $\alpha$ as follows:
$$
\alpha = \alpha_X   \hat X + \alpha_Y \hat Y + \alpha_Z  \hat Z.
$$
It follows that the space of smooth $1$-forms is identified to the product $(C^\infty(M))^3$ by the isomorphism
$$
\alpha \to  (\alpha_X, \alpha_Y, \alpha_Z) \,.
$$
By the above isomorphism, it is also possible to define Sobolev spaces of currents
$$
\Omega^s_a (M) \equiv  W^s_a(M)^3   \quad \text{ for } s\geq 0\,,
$$
and their dual spaces $\Omega^{-s} _a(M) := (\Omega^s_a(M))^*$ of currents.  

By the Sobolev embedding theorem,  for every rectifiable arc $\gamma$,  the current $\gamma \in \Omega^{-s}_a(M)$ for all $s>3/2$. Since
$D^H_{a}\in W^{-s}_a(M)$ for all $s>1/2$, all  basic currents  $B^H_a \in \Omega^{-s}_a(M)$ for all $s>1/2$. Notice that the Hilbert structure of $\Omega^{s}_a(M)$ and 
$\Omega^{-s}_a(M)$ depends on $a=(X,Y,Z)$. 

\bigskip
Let $\Pi^{-s}_{H}: \Omega^{-s}_a(M) \to \Omega^{-s}_a (H)$ denote the orthogonal projection on a single irreducible
component (of central parameter $ n \in \Z\setminus\{0\}$). 

Let $\mathcal B^{-s}_{H,a}: \Omega^{-s} (M) \to \C $ denote the orthogonal component map in the direction of  the $1$-dimensional
space of basic currents, supported on a single irreducible unitary representation.

We introduce below a crucial Diophantine condition.  Let $\delta_{\mathcal M} : \mathcal M \to \R^+ \cup\{0\}$ be the distance function (which projects to the hyperbolic
metric of curvature $-1$ on) from the base point  $\overline{id} \in {\mathcal M}$.  

For any $L >0$, let $DC(L)$ denote the set of $\overline{a} \in \mathcal M$ such that 
\begin{equation}
\label{eq:DC}
\int_0^{+\infty}  \exp[ \frac{1}{4} \delta_{\mathcal M}(g_{-t}({\overline a})) - \frac{t}{2} ]   dt   \leq L \,.
\end{equation}
Let $DC$ denote the union of the sets $DC(L)$ over all $L>0$.  By Kinchine's theorem, or the logarithmic law of geodesics, it follows that, for 
almost all $\overline a \in \mathcal M$, we have
$$
\limsup_{t\to +\infty}  \frac{ \delta_{\mathcal M}(g_{-t}  (\overline a))}{\log t} = 1.
$$
It follows immediately that the set $DC \subset \mathcal M$ has full Haar volume.

\smallskip

 The {\it Bufetov functionals} are defined for all Diophantine ${\overline a} \in DC$ as follows :

\begin{lemma}  \label{lemma:Bufetov_funct} Let ${\overline a} \in DC(L)$. For $s>7/2$ and every  rectifiable arc $\gamma$ on $M$, the limit 
$$
\hat \beta_{H} (a,\gamma)=  \lim_{t\to +\infty}    e^{-\frac{t}{2}}    \mathcal B^{-s}_{H,g_{-t}(a)} ( \gamma) \,.
$$
exists, is finite and defines a finitely-additive measure on the space of rectifiable arcs. 
There exists a constant $C''_s >0$ such that the following estimate holds:
$$
\vert \Pi^{-s}_H(\gamma) -  \hat \beta_{H} (a,\gamma) B^H_{a} \vert_{a,-s}   \leq C''_s (1+L) (1  + \int_\gamma \vert \hat Y\vert  + \int_\gamma \vert \hat Z\vert )\,.
$$
For every $L>0$, the function $\hat \beta_{H} (\cdot,\gamma)$  is continuous on $DC(L)\subset  \mathcal M$.
\end{lemma} 

\begin{proof}  
For simplicity of notation (since $H$ is fixed) we suppress the dependence on $H \subset L^2(M)$.
We will use subscript $a,t$ to denote any dependence on  $g_{-t}(a) = (X_t, Y_t, Z)$, for example
$$
\Pi^{-s}(\gamma):= \Pi^{-s}_H(\gamma), \quad  {\mathcal B}^{-s}_{a,t}  := \mathcal B^{-s} _{H, g_{-t}(a)} \, , \quad  B_{a,t}:= B^H_{g_{-t}(a)} \,.
$$

 For every $t\in \R$ we have the following splitting:
$$
\Pi^{-s}(\gamma) =  \mathcal B^{-s}_{a,t} ( \gamma) B_{a,t} +  R_{a,t}.
$$
Moreover this splitting is orthogonal in $\Omega^{-s}_{g_{-t}(a)}(M)$.  By construction, for any $h\in \R$ we have
$$
\mathcal B^{-s}_{a,t+h} ( \gamma) B_{a,t+h} +  R_{a,t+h} = \mathcal B^{-s}_{a,t} ( \gamma) B_{a,t} +  R_{a,t}.
$$
Since  by Lemma \ref{scaldist} we have $ B_{t+h} = e^{-h/2}  B_t $ it follows that
$$
 \mathcal B^{-s}_{a,t+h} ( \gamma) =  e^{h/2} \mathcal B^{-s}_{a,t} ( \gamma) +  \mathcal B^{-s}_{a,t+h} ( R_{a,t})\,.
$$
By differentiating the expression  at $h=0$, we get
$$
\frac{d}{dt}  \mathcal B^{-s}_{a,t}(\gamma) = \frac{1}{2}  \mathcal B^{-s}_{a,t}(\gamma)  + [\frac{d}{dh} \mathcal B^{-s}_{a,t+h} ( R_{a,t})]_{h=0}\,.
$$
The derivative on the right hand side of the above equation can be computed in representation.  Let $<\cdot, \cdot>_t$ denote the inner
product in the Hilbert space $\Omega^{-s}_{g_{-t}(a)}$.  From the intertwining formulas~\eqref{eq:inteq} it follows that
$$
\begin{aligned}
  \mathcal B^{-s}_{a,t+h} ( R_{a,t}) &= <R_{a,t}, \frac{B_{a,t+h}}{ \vert B_{a,t+h} \vert_{t+h}^2}>_{a,t+h} \\ &=  <R_{a,t}\circ U_{-h}, \frac{B_{a, t+h}\circ U_{-h}}{ \vert B_{a,t+h} \vert^2_{ a,t+h}} >_{a,t}
  \\  &=<R_{a,t}\circ U_{-h}, \frac{B_{a,t}}{ \vert B_{a,t} \vert^2_{a,t}} >_{ a,t} = \mathcal B^{-s}_{a,t} ( R_{a,t} \circ U_{-h})\,.
  \end{aligned}
$$
Now by the definition of the intertwining operators $U_h$ in formula~\eqref{eq:intop} it follows that, in the sense of distributions, 
$$
\frac{d}{dh}  (R_{a,t} \circ U_{-h})  = - R_{a,t} \circ (X_t +\frac{1}{2}) \circ U_{-h} = [(X_t- \frac{1}{2})R_{a,t}] \circ U_{-h}\,.  
$$
We conclude that
$$
[\frac{d}{dh} \mathcal B^{-s}_{a,t+h} ( R_{a,t})]_{h=0} = - \mathcal B^{-s}_{a,t}(  (X_t- \frac{1}{2})R_{a,t} )\,.
$$
We finally claim that the following estimate holds: for all rectifiable curve $\gamma$ on $M$ and all $t\in \R$, we have
$$
\begin{aligned}
 \vert \mathcal B^{-s}_{a,t}(  (X_t- \frac{1}{2})R_{a,t}(\gamma) ) \vert &\leq \vert  { R_{a,t}} (\gamma) \vert_{g_{-t}(a), -(s+1)} \\ &\leq C_s \exp [\frac{1}{4} \delta_\mathcal M (g_{-t}(a))]  (1 + \int_\gamma \vert \hat Y\vert  + \int _\gamma \vert \hat Z\vert )\,.
 \end{aligned}
$$
The above remainder estimate will be proved in the lemma below.
We get therefore a scalar differential equation
$$
\frac{d}{dt}  \mathcal B^{-s}_{a,t}(\gamma) = \frac{1}{2}  \mathcal B^{-s}_{a,t}(\gamma) + \mathcal R_{a,t}(\gamma)
$$
with a bounded non-negative function $\mathcal R_{a,t}(\gamma)$ satisfying the estimate
\begin{equation}
\label{eq:DC_bound}
\mathcal R_{a,t}(\gamma) \leq  C_s \exp [\frac{1}{4} \delta_\mathcal M (g_{-t}(a))]  (1 + \int_\gamma \vert \hat Y\vert  + \int _\gamma \vert \hat Z\vert )\,.
\end{equation}
The solution of the above differential equation is
$$
\mathcal B^{-s}_{a,t}(\gamma) = e^{\frac{t}{2}} [ \mathcal B^{-s}_{a,0}(\gamma) + \int_0^t  e^{-\frac{\tau}{2}} \mathcal R_{a,\tau}(\gamma) d\tau ]\,.
$$
It follows that, under the Diophantine assumption that ${\bar a} \in DC(L)$,
$$
\lim_{t\to +\infty}   e^{-\frac{t}{2}}  \mathcal B^{-s}_{a,t}(\gamma) = \hat \beta_{H}(a,\gamma)
$$
exists. Since by definition the distributions $\mathcal B^{-s}_{a,t}(\gamma)$ and $\mathcal R_{a,t}(\gamma)$ depend continuously on
$(a,t) \in A\times\R$, by the Diophantine bound~\eqref{eq:DC_bound}, which implies the convergence of the integral 
$$
\int_0^{+\infty}  e^{-\frac{\tau}{2}} \mathcal R_{a,\tau}(\gamma) d\tau \,,
$$
it follows that the complex number
$$
\hat \beta_{H}(a,\gamma) = \mathcal B^{-s}_{a,0}(\gamma) + \int_0^{+\infty} e^{-\frac{\tau}{2}} \mathcal R_{a,\tau}(\gamma) d\tau
$$
depends continuously on $a \in DC(L)$.  Moreover, we have 
$$
\Pi^{-s}_{H}(\gamma) -  \hat \beta_{H} (a,\gamma) B^H_{a} =R_0-\left(\int_0^{+\infty}  e^{-\frac{\tau}{2}} \mathcal R_{a,\tau} d\tau \right)B^H_{a}
$$
and by the above bound on the remainder terms $\mathcal R_{a,t}$ and by the Diophantine condition on ${\overline a} \in A$, it follows that
$$
\vert\Pi^{-s}_{H,a}(\gamma) -  \hat \beta_H (a,\gamma) B^H_{a}\vert _{a,-s}\leq C''_s (1+L) (1 + \int_\gamma \vert \hat Y\vert  + \int _\gamma \vert \hat Z\vert ).
$$

The argument is thus concluded, up to the above claim on the remainder bounds.

\end{proof}

We then prove the claim on the remainder bounds.
\begin{lemma} There exists $C_s>0$ such that, for all $t\geq 0$ and all rectifiable arcs $\gamma$ we have
$$
\vert  R_t(\gamma) \vert_{a,-s} \leq C_s \exp [\frac{1}{4} \delta_\mathcal M (g_t(a))]  (1 + \int_\gamma \vert \hat Y\vert  + \int _\gamma \vert \hat Z\vert )\,.
$$
\end{lemma}
\begin{proof}
Let $\alpha$ be any $1$-form. For simplicity, for all $t\in \R$, we let $g_t(a) = (X_t, Y_t, Z)$. We can write
$$
\alpha = \alpha_{X_t} \hat X_t  +  \alpha_{Y_t} \hat Y_t + \alpha_{Z} \hat Z\,.
$$
Let us assume now that $\alpha$ is supported on a single irreducible component $H$.  Since
$$
\omega = \hat X_t \wedge \hat Y_t \wedge \hat Z\,,  \quad  \text{hence } \eta_{X_t} =  \hat Y_t \wedge \hat Z\,,
$$
we have the identity
$$
B^H_t (\alpha) = D^H_t (  \alpha_{X_t}  \eta_{X_t} \wedge  \hat {X_t} +  \alpha_{Y_t} \eta_{X_t} \wedge \hat {Y_t} + \alpha_Z \eta_{X_t} \wedge \hat Z)
=  D^H_t (  \alpha_{X_t})\,.
$$
Let us then assume that
$$
B^H_t (\alpha) = D^H_t (  \alpha_{X_t}) =0\,.
$$
It follows that $\alpha_{X_t}$ is a coboundary for the cohomological equation, that is, there exists a smooth function
$u$ on $M$ (with a loss of Sobolev regularity of $1+$) such that
$$
\alpha_{X_t} = X_t u \,.
$$
By the Sobolev embedding theorem, for any $ s> r+1 > 7/2$,  there exists a constant $B_r(g_t(a)$ we have
$$
\vert u \vert_{C^0(M)} +  \vert Y_tu \vert_{C^0(M)} + \vert Z u \vert_{C^0(M)}     \leq   B_r(g_t(a))  \vert  u \vert_{g_t(a), r} \leq   B_r(g_t(a))  \vert  \alpha_{X_t}  \vert_{g_t(a),s}\,.
$$
By \cite{FlaFo}, Corollary 3.11, there exists a universal constant $C_r>0$ such that the {\it best Sobolev constant }$B_r(a)$ is bounded above as follows:
\begin{equation}
\label{eq:best_Sob_const}
B_r(a) \leq   C_r \exp [\frac{1}{4} \delta_\mathcal M (a)] \,.
\end{equation}
We remark that we can write
$$
du = (X_tu)\, \hat X_t + (Y_tu)\, \hat Y_t + (Zu)\, \hat Z\,, 
$$
hence, by the Sobolev embedding theorem and the fact that $\hat Y_t=e^{-t}\hat Y$, for all $s >7/2$, we have
$$
\begin{aligned}
\vert \int_\gamma \alpha  \vert &= \vert \int_\gamma du + (\alpha_{Y_t} - Y_tu) \hat Y_t + (\alpha_Z -Zu) \hat Z \vert
\\ &\leq  C_s \vert \alpha\vert_{g_t(a), s}  \exp [\frac{1}{4} \delta_\mathcal M (g_t(a))]   (1 +   \int_\gamma \vert  \hat Y    \vert  +  \int_\gamma \vert  \hat Z    \vert)\,.
\end{aligned}
$$
Let us now consider an arbitrary smooth $1$-form $\alpha$ on $M$ supported on a single irreducible component. There exists a orthogonal decomposition
$$
\alpha =  \alpha_0 +   \alpha_0^\perp  \in \Omega^s_{g_t(a)}(M)
$$
such that $\alpha_0 \in \text{Ker} (B^H_t)$. Since $R_t (\gamma) \in \{ B^H_t\}^\perp \in \Omega^{-s}_{g_t(a)}(M)$, 
and $$ \alpha_0^\perp \in [\text{Ker}(B^H_t)]^\perp=\text{Ker} (\{B^H_t\}^\perp)  \in \Omega^s_{g_t(a)}(M)\,,$$
it follows that 
$$
R_t (\gamma) (\alpha) =  R_t (\gamma) (\alpha_0) = \int_\gamma \alpha_0\,,
$$
hence the above estimate leads to the bound
$$
\vert R_t (\gamma) (\alpha) \vert \leq C_s \vert \alpha_0
\vert_{g_t(a),s} \exp [\frac{1}{4} \delta_\mathcal M (g_t(a))]   (1 +   \int_\gamma \vert  \hat Y    \vert  +  \int_\gamma \vert  \hat Z    \vert)\,.
$$
The conclusion immediately follows by the orthogonality of the decomposition.

\end{proof}

\section{Main properties of the functionals}
\label{sec:Main_properties}

By definition, the Bufetov functional has the {\it additive property}, that is, for all rectifiable arcs $\gamma_1$ and
$\gamma_2$ on $M$, by linearity of projections and limits, we have
\begin{equation}
\label{eq:additive}
\hat \beta_H (a, \gamma_1 + \gamma_2 ) =  \hat \beta_H (a,\gamma_1) + \hat \beta_H(a,\gamma_2)\,;
\end{equation}
it has the {\it scaling property}, that is, for every rectifiable arc $\gamma$ and $t>0$, we have
\begin{equation}
\label{eq:scaling}
\hat \beta_H  ( g_t(a), \gamma) =  e^{-t/2}  \hat \beta_H (a, \gamma)\,, 
\end{equation}

The Bufetov functional also has the following {\it invariance property}: for all rectifiable arc $\gamma$  and for all $\tau>0$,
\begin{equation}
\label{eq:invariance}
\hat \beta_H( a, (\phi_\tau^{Y})_* (\gamma) )  =  \hat \beta_H(a,\gamma)\,.
\end{equation}
The above invariance property follows from the fact that by the Sobolev embedding theorem we have
$$
\vert  \phi^Y_\tau (\gamma)) - \gamma \vert_{g_{-t}(a), -s}  \leq  C\tau   B_s(g_{-t}(\overline{a})) \leq C\tau \exp [ \frac{\delta_{\mathcal M}(g_{-t}(\overline a))}{4}]  \,.
$$
In fact, the current $\phi^Y_\tau (\gamma)) - \gamma$ is equal, up to two bounded orbit arcs of the flow $\phi^Y_\R$,  to the boundary of a $2$-dimensional 
current $\Delta$, which has uniformly bounded $2$-dimensional area with respect to the frame $g_{-t}(a):=(X_t, Y_t, Z)$. This follows from the fact that $\Delta$
can be taken to be surface tangent to the flow $\phi^Y_\R$, hence, not only 
$\hat X_t\wedge \hat Y_t = \hat X\wedge \hat Y  $,  but also
$$
\int_{\Delta}   \vert \hat Y_t  \wedge \hat Z_t \vert  =   e^{-t}  \int_0^\tau [\int_{\phi^Y_{\sigma}(\gamma)}  \vert \hat Z \vert] d\sigma \quad \text{ and } \quad  \int_{\Delta}    
\vert \hat X_t  \wedge \hat Z_t \vert =0  \,.
$$

The above invariance property then follows immediately from the Diophantine condition and from the existence of the Bufetov functional.

Finally, the Bufetov functional has the following {\it vanishing property}: for every rectifiable arc $\gamma$ {\it tangent
to the central-stable foliation} of the extended renormalization flow $\hat g_\R$ on the fibers of the extended moduli space $\hat {\mathcal M}$, that is, the foliation generated 
by the integrable distribution $\{Y,Z\}$ above each point $a=(X,Y,Z)$, we have
\begin{equation}
\label{eq:vanishing}
\hat \beta_H(a,\gamma) =0 \,.
\end{equation}
The vanishing property is a direct consequence of the definition, as the length of any arc $\gamma$ tangent to the central-stable foliation is uniformly bounded along the
backward orbit of the renormalization flow:
$$
\int_\gamma \vert \hat X_t \vert =0, \quad \int_\gamma \vert \hat Y_t \vert = e^{-t}  \int_\gamma \vert \hat Y \vert\,, \quad 
\int_\gamma \vert \hat Z_t \vert =\int_\gamma \vert \hat Z \vert\,.
$$

\medskip
Let now $\gamma^X_T (x)$ denote the arc of orbit of the flow $\phi^X_\R$, that is,
$$
\gamma^X_T (x) = \{ \phi^X_t (x) \vert   t\in [0,T]\}\,,
$$
and, for all $(x,T)\in M \times \R^+$, let 
$$
\beta_H(a, x, T) :=  \hat \beta_H \left( a, \gamma^X_T (x)\right) \,.
$$
From  the additive property we derive the following {\it cocycle property}: for all $(x,T_1, T_2)\in M\times \R\times \R$ we have
$$
\beta_H(a, x, T_1+ T_2) = \beta_H (a,x,T_1) + \beta_H(a, \phi^X_{T_1}(x), T_2)\,.
$$
Moreover, for $\alpha\in A_\Gamma$, we have 
\begin{equation}\label{equivar}
\beta_H (\alpha a , \alpha(x),T)=\beta_H (a,x,T),
\end{equation}
which means that the function $\beta_H(\cdot, \cdot,T)$ is a well defined function on the extended moduli space $\hat {\mathcal{M}}$.  
By Lemma~\ref{lemma:Bufetov_funct} for any smooth function $f$ which belongs to a single irreducible component $H$, we have
\begin{equation}\label{appro}
\begin{aligned}
\vert &\int_0^T f \circ \phi^X_t (x) dt - \beta_H  (a,x,T) D_{a}^H (f)  \vert   \\
&=\vert <\gamma^X_T (x), f \hat X> -  
\hat \beta_H  \left(a, \gamma^X_T (x)\right) B_{a} ( f\hat X)\vert \leq C''_s (1+L) \vert f \vert_{a,s}\,.
\end{aligned}
\end{equation}

By the scaling property of the Bufetov functional and \eqref{equivar} we derive the following scaling identities: for all $(x,T)\in M\times \R^+$ and $t\in\R$, we have
\begin{equation}\label{buf:scaling}
\beta_H (a, x,Tt) = T^{1/2} \beta_H(g_{\log T}(a), x,t) = T^{1/2}  \beta_H(\hat g_{\log T}([a, x]_{A_\Gamma}),t)   \,.
\end{equation}
We have therefore derived the following asymptotic formula for ergodic averages. For every $x\in M$ and $t$, $T>0$ we have
\begin{equation}\label{orbin}
\left\vert \int_0^{tT}  f \circ \phi^X_\tau (x) d\tau - T^{1/2} \beta_H \left( {\hat g}_{\log T}[a, x]_{A_\Gamma}, t\right) D^H_{a} (f)  \right\vert \leq  C''_s(1+L) \vert f \vert_{a,s}\,.
\end{equation}

As an immediate consequence of the above asymptotic property we can derive the following {\it orthogonality property} : for all $a\in DC$
and for all $t\in \R^+$, for any smooth function $f\in H$ we have 
\begin{equation}
\label{orth}
\beta_H(a, \cdot, t) \in H \subset L^2(M)  \,.
\end{equation}
In fact, by equations \eqref{appro} and \eqref{buf:scaling} we have
\begin{equation}
\begin{aligned}
\label{buf_lim}
D^H_a(f) \beta_H(a,x ,t) &= \lim_{T\to +\infty} \frac{1}{T^{1/2}}  \beta_H (g_{-\log T}(a) , x ,Tt)  \\ &= \lim_{T\to +\infty}   \frac{1}{T^{1/2}} \int_0^{Tt} f \circ \phi^{X_ {g_{-\log T}(a)} }_\tau (x) d\tau\,.
\end{aligned}
\end{equation}
It follows that $\beta_H(a, \cdot, t)\in H$ as a pointwise {\it uniform} limit  of (normalized) ergodic integrals functions of any given function $f\in H$.

It can also be proved (as in the work of Bufetov \cite{Bu}, or~\cite{BF}) that the Bufetov functionals are H\"older for exponent $1/2-$ along the orbits of the flow 
$\phi^X_\R$.  

In fact, the {\it H\"older property} for the Bufetov functionals on
rectifiable arcs takes the following form: there exists a constant $C>0$ such that, for every (admissible) rectifiable arc $\gamma$
on $M$ we have 
$$
\vert \hat \beta_H(a,\gamma) \vert \leq   C \left( 1+ \int_{\gamma} \vert \hat X\vert +  \int_{\gamma} \vert \hat Y \vert +\int_\gamma \vert \hat Z\vert \right) (  \int_{\gamma} \vert \hat X\vert)^{1/2} \,.
$$
The H\"older property is an immediate consequence of the scaling property and of uniform bounds for the Bufetov functionals
on arcs of bounded length. From the above property we can easily derive the H\"older property for the H\"older cocycles $\beta(a, x,T)$ with respect to $x\in M$ along the orbits of the flow $\phi^X_\R$ or with respect to the time $T\in \R$. 
We conclude this section by constructing Bufetov functionals of smooth functions.
By the theory of unitary representations we can write
$$
L^2(M)=\bigoplus_{n\in \Z} H_n :=  \bigoplus_{n\in \Z} \bigoplus_{i=1}^{\mu(n)} H_{i,n}\,,
$$
where $H_n= \bigoplus_{i=1}^{\mu(n)} H_{i,n}$ is the space with central parameter $n\in \Z\setminus \{0\}$, and $H_{i,n}$  are irreducible representation spaces, for $i=1,\dots, \mu(n)$.  It follows from the Howe-Richardson multiplicity formula (or by a direct calculation of irreducible representations) that 
$$
\mu(n) = \vert n \vert   \,, \quad \text{ for all }\, n\not =0\,.
$$
For every $n\not=0$ and every $i\in \{1, \dots,\mu(n)\}$, let $D^{i,n}_a$ denote the unique normalized $X_a$-invariant distribution  supported on $W^{-s}(H_{i,n})$
and $\beta^{i,n}= \beta_{H_{i,n}}$ the associated Bufetov functional. Since any function $f\in W^s(M)$ has a decomposition 
$$
f= \sum_{n\in \Z} \sum_{i=1}^{\mu(n)}  f_{i,n} 
$$
where each component $f_{i,n} \in W^s(H_{i,n})$, we can define the Bufetov cocycle associated to the function $f\in W^s(M)$ as the sum
\begin{equation}
\label{buf_f}
\beta^f (a,x,T) :=   \sum_{n\in \Z} \sum_{i=1}^{\mu(n)}  D^{i,n}_a (f) \beta^{i,n}(a,x,T)\,.
\end{equation}
The following result is a version for Bufetov functionals of bound on ergodic integrals proved in \cite{Fo2}, Lemma 1.4.9.
For every $(a,T) \in A\times \R^+$ ,  we introduce the excursion function
\begin{equation}
\label{eq:exc}
\begin{aligned}
E_{\mathcal M} (a, T) := & \int_0^{\log T}   \exp \left( \frac{\delta_{\mathcal M} (g_{\log T-t}(\bar a))}{4} -\frac{t}{2}\right) dt  \\
& =     T^{-1/2} \int_0^{\log T}   \exp \left( \frac{\delta_{\mathcal M} (g_{t}(\bar a))}{4} +\frac{t}{2}\right) dt  \,.
\end{aligned}
\end{equation}

\begin{lemma}  \label{lemma:convergence} For all Diophantine $a \in DC(L)$ and  for all function $f\in W^s(M)$ for $s>2$, the Bufetov functional  $\beta^f $ is 
defined by a uniformly convergent series, hence  the function $\beta^f_a$ is a H\"older function  on $M\times \R$. In addition there exists a constant $C_{s}>0$ such that 
whenever $a\in DC(L)$ we have, for all $(x,t,T) \in M\times (\R^+)^2$, 
$$
\vert \beta^f (a,x,tT) \vert \leq  C_{s} \left(L+ T^{1/2} (1+ t + E_{\mathcal M}(a,T)) \right) \vert  f \vert_{a,s}\,.
$$
\end{lemma} 
\begin{proof} It follows from Lemma~\ref{lemma:Bufetov_funct}   that there exists a constant $C>0$ such that whenever $a\in DC(L)$ then
\begin{equation}
\label{eq:buf_comp_bound}
\vert \beta^{i,n}(a,x,t) \vert  \leq  C (1+L+t)\,, \quad \text{ for all }  (x,t)\in M\times \R^+\,.
\end{equation}
By the exact scaling property in formula~\eqref{buf:scaling}, we have
$$
\beta^{i,n} (a, x,tT) = T^{1/2} \beta^{i,n} (g_{\log T}(a), x,t) =   T^{1/2} \beta^{i,n} (\hat g_{\log T}([a, x]_{A_\Gamma}),t) \,.
$$
By the definition of the set $DC(L)$ in formula~\eqref{eq:DC} we have that whenever  $a\in DC(L)$ then $g_{\log T}(a) \in DC(L_T)$  with
\begin{equation}
\label{eq:L_T} 
L_T  \leq   L  T^{-1/2} + E_{\mathcal M} (a, T) \,,
\end{equation}
hence by the  bound in formula~\eqref{eq:buf_comp_bound} we have that
$$
\vert \beta^{i,n}  (g_{\log T}(a), x,t)  \vert  \leq  C (1+L_T+t)\,, \quad \text{for all } (x,t, T)\in M\times (\R^+)^2\,.
$$
It follows that for all $r>1/2$ we have
$$
\begin{aligned}
\vert &\beta^f (a,x,tT) \vert \leq  C _r  T^{1/2} (1+L_T+t) \sum_{n\in \Z} \sum_{i=1}^{\mu(n)}  \vert f_{i,n} \vert_{a,r}  \\ &\leq  C_r T^{1/2} (1+L_T+t)
 (\sum_{n\in \Z} (1+n^2)^{-r' })^{-1/2}    ( \sum_{n\in \Z} [ \sum_{i=1}^{\mu(n)}  \vert (1-Z^2)^{r'/2}  f_{i,n} \vert_{a,r}]^2)^{1/2}\,.
\end{aligned}
$$
We therefore conclude that for all $r'>3/2$ there exists a constant $C_{r,r'}>0$ such that
$$
\vert \beta^f (a,x,tT) \vert \leq  C_{r,r'} T^{1/2} (1+L_T+t) \vert  f \vert_{a, r+r'}\,,
$$
hence, in view of formula~\eqref{eq:L_T},  the statement is proved.
\end{proof}

It follows from the convergence result given in Lemma~\ref{lemma:convergence} that all properties of the Bufetov functionals $\beta_H$, each associated to a single
irreducible component,  extend to the Bufetov functionals $\beta^f$ for any $f\in W^s(M)$ ($s>2$). 

In particular, for every Diophantine $a=(X,Y,Z) \in DC$ the function 
$\beta^f_a$ on $M\times \R$ is a H\"older cocycle for the flow $\phi^X_\R$, which satisfies the scaling property~\eqref{buf:scaling}, that is, for all $(x,t,T)\in M\times (\R^+)^2$, we have
\begin{equation} \label{eq:scaling_general}
\beta^f (a, x,Tt) = T^{1/2} \beta^f (g_{\log T}(a), x,t) =   T^{1/2} \beta^f (\hat g_{\log T}([a, x]_{A_\Gamma}),t)   \,.
\end{equation} 

Finally from the asymptotic formula  \eqref{orbin} on each irreducible component we derive the following asymptotic result:
\begin{theorem} \label{thm:asymptotic} For all $s>7/2$ there exists a constant $C_{s}>0$ such that for all $a =(X,Y,Z)\in DC(L)$,  for all $f\in W^s(M)$
and for all $(x,T)\in M\times \R^+$,  we have
\begin{equation}
\left\vert \int_0^{T}  f \circ \phi^X_t (x) dt - \beta^f (a, x, T)   \right\vert \leq  C_{s} (1+L) \vert f \vert_{a,s}\,.
\end{equation}
\end{theorem} 
All the results of this paper, about limit distributions and about decay of correlations for time-changes, are derived from the above asymptotic result.

\section{Limit distributions}
 \label{sec:limit_dist}

In this section we derive some corollaries on limit distributions of ergodic integrals, which generalize results of J.~Griffin and J.~Marklof \cite{GM} to arbitrary smooth
functions and recover the H\"older property of limit distributions, proved by F.~Cellarosi and J.~Marklof \cite{CM}.

\begin{lemma} 
\label{lemma:L2lim}
There exists a continuous modular function $\theta_H:A \to H \subset L^2(M)$ such that  for any  
$f  \in W^s_a(H)$ with $s>1/2$,  we have
$$
 \lim_{T\to +\infty}   \Vert  \frac{ 1} {T^{1/2}}  \int_0^{T}  f\circ \phi^X_t(\cdot) dt - \theta_H\left(g_{\log T}(a)\right) D^H_a(f)  \Vert_{L^2(M)}  =0\,.
$$
The family $\{\theta_H(a)\vert a\in A\}$ has (positive) constant norm in $L^2(M)$: there exists a constant $C>0$ such that  for all $a \in A$ we have
$$
\Vert  \theta_H (a) \Vert_{L^2(M)} = C \,.
$$ 
\end{lemma}
\begin{proof}
We refer the reader to~\cite{FlaFo} or \cite{Fo2} for background on the application of this theory to the cohomological equation of Heisenberg nilflows. By the Stone-Von Neumann theorem, the space $H:=H_z$ is unitarily equivalent to the space $L^2(\R, du)$ and such a unitary equivalence can be chosen so that the group $\phi^X_\R$ is represented as the group of translations on the real line and the group $\phi^Y_\R$ is a group of the form $\{e^{\imath z t} \text{Id}\}$. In other terms we have the infinitesimal representation
$$
X\to \frac{d}{du} \,, \qquad   Y \to \imath  z u\,.
$$
The space of smooth vectors (for a irreducible unitary representation of central parameter $z\not=0$) is the Schwartz space $\mathcal S(\R)$ and the space of translation invariant tempered distribution on the real line is given by all scalar multiples of the Lebesgue measure. It follows that  invariant distributions 
supported on $H$ for the flow $\phi^X_\R$  are represented as scalar multiples of the Lebesgue measure.  Finally the Sobolev space $W^s_a(H)$ is represented
as the space ${\mathcal S}^s(\R)$ of functions $ f \in L^2(\R, du)$ such that
$$
\int_\R  \vert  (1+ \frac{d^2}{du^2} + z^2 u^2)^{s/2} \hat f(u) \vert^2 du  \, <\, +\infty\,.
$$
The statement is equivalent to the claim that there exists $\theta(a)\in L^2(\R, du)$ such that for all $ f\in {\mathcal S}^s(\R)$ we have
$$
 \lim_{T\to +\infty}   \Vert    \frac{ 1} {T^{1/2}} \int_0^{T}  f (u+t) dt - \theta \left(g_{\log T}(a)\right) \text{Leb}(f)  \Vert_{L^2(\R,du)}  =0\,.
$$
An equivalent formulation, by the standard Fourier transform on $\R$:
$$
 \lim_{T\to +\infty}  \Vert   \frac{ 1} {T^{1/2}}  \int_0^{T}  e^{i t \hat u}  \hat f (\hat u) dt- \hat \theta\left(g_{\log T}(a)\right) \hat f(0)  \Vert_{L^2(\R,d\hat u)}  =0\,.
$$
Let $\chi \in L^2(\R, d\hat u)$ denote the function defined as
$$
\chi (\hat u) =   \frac{ e^{\imath \hat u} -1 }{i \hat u}   \,, \quad \text{ for all } \hat u \in \R\,.
$$
Let $\hat \theta (a)(\hat u) :=   \chi (\hat u)$, for all $\hat u\in \R$.   Let us compute $\theta\left(g_{\log T}(a)\right)$. By definition  $g_{\log T}(a) = (T X, T^{-1} Y, Z)$, hence
the induced representation
$$
TX \to   i T \hat u \,,  \qquad  T^{-1}  Y  \to   T^{-1} z \frac{d}{d\hat u}
$$
is intertwined to the normalized representation by the unitary equivalence $U_T:L^2(\R)\to L^2(\R)$ defined as
$$
U_T (f)(\hat u) = T^{1/2} f(T \hat u) \,,  \quad \text{ for all }  u\in \R\,.
$$
In fact, we have 
$$
i \hat u  =  U_T^{-1}  \circ  (i T \hat u )\circ U_T    \,, \quad     z  \frac{d}{d\hat u} =   U_T ^{-1} \circ (T^{-1} z \frac{d}{d\hat u} ) \circ U_T \,.
$$
It follows that, for all $a\in A$ and all $T>0$, 
$$
\hat \theta \left(g_{\log T}(a)\right)(\hat u) =  U_T (\chi) (\hat u) = T^{1/2} \chi( T\hat u)\,.
$$
The function $\theta_H(a)\in H$ is uniquely defined in representation as the unique Fourier anti-transform $\theta(a)\in L^2(\R)$ of the function $\hat \theta (a)\in L^2(\R)$. 
By its definition the function $\theta_H$ is modular, that is, it is invariant under the action of the lattice $A_\Gamma$ on $A$. As a consequence, it induces a well-defined
function on the moduli space ${\mathcal{M}} = A_\Gamma\backslash A$.
By unitary equivalence 
$$
\Vert \theta_H(a) \Vert_H =  \Vert \theta (a) \Vert_{L^2(\R)} =  \Vert \hat \theta (a) \Vert_{L^2(\R)} = \Vert   \frac{ e^{i \hat u} -1} {\imath \hat u} \Vert_{L^2(\R, d\hat u)} := C>0\,.
$$

By integration we have
$$
\begin{aligned}
 \int_0^{T}  e^{\imath \hat u t} \hat f (\hat u) dt &= T \chi (T\hat u) \hat f (u) \\ &= T \chi (T\hat u) \left(\hat f (\hat u) - \hat f(0)\right)  +  
 T^{1/2} \hat \theta\left(g_{\log T}(a)\right)(\hat u)  \hat f(0)   \,.
\end{aligned}
$$
The claim is therefore reduced to the following statement
$$
\limsup_{T\to +\infty}    \Vert    T^{1/2} \chi (T\hat u) \left(\hat f (\hat u) - \hat f(0)\right) \Vert_{L^2(\R,d\hat u)}  \,=\,0\,.
$$
Since by hypothesis $f\in \mathcal S^s(\R)$ with $s>1/2$,  the function $\hat f \in C^0(\R)$ and it is bounded, hence
$$
\Vert  T^{1/2} \chi (T\hat u) \left(\hat f (\hat u) - \hat f(0)\right) \Vert_{L^2(\R,d\hat u)}  = \Vert   \chi (v) \left(\hat f (\frac{v}{T}) - \hat f(0)\right) \Vert_{L^2(\R,dv)} \to 0\,,
$$
by the dominated convergence theorem.

\smallskip The continuous dependence of the the function $\theta_H(a)\in H$ on $a\in A$ follows from the continuous dependence of the intertwining operator
$U_a:L^2(\R) \to L^2(\R)$, which conjugates a representation of the form
$$
X_a  \to   \alpha \frac{d}{du} + \imath {\gamma z }u + v \,, \qquad  Y_a \to  \beta \frac{d}{du} + \imath \delta z  u + w
$$
to the standard representation, on the parameters
$$
\begin{pmatrix} \alpha & \beta  \\ \gamma & \delta   \end{pmatrix}  \in SL(2, \R) \quad \text{ and }  \quad (v,w) \in \R^2\,,
$$
of the automorphism $a\in A$.  The intertwining operator $U_a$ can be computed explicitly. The details are left to the reader.

\end{proof}

\begin{corollary} 
\label{cor:L2lim}
 There exists a constant $C>0$ such that,  for any  $s>1/2$, for any $a=(X,Y,Z)\in A$ and for any  $f  \in W^s_a(H)$,  we have
$$
\lim_{T\to +\infty}    \frac{ 1} {T^{1/2}}  \Vert \int_0^T f\circ \phi^X_t dt
\Vert_{L^2(M)} = C  \vert D_a^H (f ) \vert  \,.
$$
\end{corollary}
 The above statement strengthens Lemma 15 of~\cite{AFU}. In fact, there the authors proved a slightly weaker statement for linear skew-shifts of the
torus $\T^2$, that is for maps of the form 
$$
T_{\rho, \sigma} (y,z) =   (y+\rho, z+y+ \sigma) \, , \quad \text{ for all }  (y,z)\in \T^2\,.
$$
As explained in \cite{AFU}, the minimal flow $\phi^X_\R$ has constant return time to a transverse torus with return map a linear skew-shift, that is, a map 
of the form $T_{\rho, \sigma}$ for constants $\rho\in \R \setminus \Q$ and $\sigma \in \R$. 

From Corollary~\ref{cor:L2lim}  we derive the following limit result for the $L^2$ norm of Bufetov functionals.
\begin{corollary} \label{cor:L2bounds_buf}   There exists a constant $C>0$ such that for all  irreducible components  $H\subset L^2(M)$  and
 $a\in DC$,  we have 
$$
\lim_{T\to +\infty}    \frac{ 1} {T^{1/2}}    \Vert \beta_H (a, \cdot, T) \Vert_{L^2(M)}= C \,.
$$
\end{corollary}
\begin{proof} By the normalization of the invariant distributions in the Sobolev space $W^s(M)$ for any given $s>1/2$
the following holds. For all irreducible components~$H$ and all $a\in A$ , there exists a (non-unique) function $f^H_a \in W^s(H)$ such that 
$$
D^H_a (f^H_a) =  \Vert  f^H_a \Vert_s= 1\,.
$$ 
For all $a\in DC(L)$, and for $s>7/2$, we derive from the asymptotic formula in Theorem~\ref{thm:asymptotic} that
$$
\left\vert \int_0^{T}  f^H_a \circ \phi^X_t (x) dt - \beta_H (a, x, T)   \right\vert \leq  C_{s} (1+L)\,.
$$
The $L^2$ estimates in the statement then follow from Corollary~\ref{cor:L2lim}.
\end{proof}

\medskip
A relation between the Bufetov functionals and the above theta functionals is established below.
\begin{corollary} \label{cor:buf_theta} For any irreducible component $H\subset L^2(M)$ the following holds. For any $L>0$ and for any $g_\R$-invariant probability measure  $\mu$ supported on $DC(L) \subset \mathcal M$, we have
$$
\beta_H ( a, \cdot, 1) =  \theta_H (a) \in L^2(M)\,,  \quad \text{ for $\mu$-almost all } \bar a\in \mathcal M\,.
$$
\end{corollary}
\begin{proof}
By Theorem~\ref{thm:asymptotic} and Lemma~\ref{lemma:L2lim}  we immediately derive that there exists a constant $C_\mu>0$ such that  for all $a\in \text{supp}(\mu) \subset DC(L)$, for all $T>0$ we have
\begin{equation}
\Vert \beta_H (g_{\log T}(a), \cdot, 1) - \theta_H (g_{\log T }(a))\Vert_{L^2(M)} \leq   \frac{C_\mu}{T^{1/2}} \,.
\end{equation}
By Luzin's theorem, for any $\delta>0$ there exists a compact subset $\mathcal E(\delta)\subset \mathcal M$ such that we have the measure bound $\mu (\mathcal M \setminus \mathcal E(\delta)) < \delta$ and such that the function  $\beta_H (a, \cdot, 1)\in L^2(M)$ depends continuously on $\bar a\in \mathcal E(\delta)$. By Poincar\'e recurrence there is a full measure subset ${\mathcal E}'(\delta) \subset \mathcal E(\delta)$ of $g_\R$-recurrent points. In particular, for every $\bar a_0\in {\mathcal E}'(\delta)$ there is a diverging sequence $(t_n)$  such that $\{ g_{t_n} (\bar a_0) \} \subset \mathcal E(\delta)$ with
$$
\lim_{n\to \infty} g_{t_n} (\bar a_0) = \bar a_0\,.
$$
By the continuity of the function $\theta_H:\mathcal M \to L^2(M)$, by the continuity at $\bar a_0$ of the function $\beta_H (\bar a, \cdot, 1)\in L^2(M)$ on the set $\mathcal E(\delta)$, and by the above $L^2$ estimate, we have
$$
\Vert  \beta_H (\bar a_0, \cdot, 1)- \theta_H(\bar a_0) \Vert_{L^2(M)}  = \lim_{n\to \infty}  \Vert \beta_H (g_{t_n}(\bar a_0), \cdot, 1)- \theta_H(g_{t_n}(\bar a_0)) \Vert_{L^2(M)}=0\,.
$$
We have thus proved that $\beta_H (\bar a, \cdot, 1)= \theta_H(a) \in L^2(M)$ for all $a\in {\mathcal E}'(\delta)$. 
It follows that the set where the latter identity fails has $\mu$-measure less than any $\delta >0$, hence the identity holds for $\mu$-almost all $\bar a\in A$, as stated.

\end{proof}

\begin{corollary} \label{cor:L2bounds_buf_bis} There exists a constant $C>0$ such that for all irreducible components $H\subset L^2(M)$ the following holds.
For any $L>0$ and for any $g_\R$-invariant probability measure  $\mu$ supported on $DC(L) \subset \mathcal M$, and for all $T>0$ we have
$$
\Vert \beta_H (a, \cdot, T) \Vert_{L^2(M)} = C T^{1/2}  \quad \text{ for $\mu$-almost all } \bar a \in \mathcal M\,.
$$
\end{corollary}

\begin{remark} Since by Lemma~\ref{lemma:L2lim}  the function $\theta_H: A\to L^2(M) $ is continuous and approximates ergodic integrals, 
it is possible to write it (at least for the first irreducible component) in terms  of  the theta function $\Theta_\chi$  introduced in \cite{GM}, as both functions are 
continuous, modular, and provide a similar asymptotic formula for ergodic averages (sums)   (see  formula (13) in~\cite{GM}).  It  follows that Bufetov functional
$\beta_H$ (for the first irreducible component) essentially coincide with the function  $\Theta_\chi$ almost everywhere on the moduli space.  The main advantage
of the approach of Cellarosi and Marklof \cite{CM} is that it provides an explicit Diophantine condition which describes the set where the function $\Theta_\chi$ is
absolutely convergent and $1/2$-H\"older (see \cite{CM}, Theorem 3.10).
\end{remark}

\smallskip
For general smooth functions we proceed as in the previous section. Since any function $f\in W^s(M)$ has a decomposition 
$$
f= \sum_{n\in \Z} \sum_{i=1}^{\mu(n)}  f_{i,n} 
$$
where each component $f_{i,n} \in W^s(H_{i,n})$, we can define the functional $\theta^f:A \to L^2(M)$ associated to the function $f\in W^s(M)$ as the weighted sum
over all functionals $\theta^{i,n} :=\theta_{H_{i,n}}$ associated to the irreducible representations $H_{i,n}$:
\begin{equation}
\label{theta_f}
\theta^f (a) :=   \sum_{n\in \Z} \sum_{i=1}^{\mu(n)}  D^{i,n}_a (f) \theta^{i,n}(a)\,.
\end{equation}
\begin{lemma}  \label{lemma:theta_convergence} For all  $a \in A$ and  for all function $f\in W^s(M)$ for $s>2$, the  functional  $\theta^f$ is 
defined by a convergent series, hence  the function $\theta^f (a)$ is an $L^2$ function  on $M$ and $\theta^f: A \to L^2(M)$ is a continuous
function.
\end{lemma} 
From Lemma~\ref{lemma:L2lim} we derive a general asymptotic theorem:

\begin{theorem} \label{thm:theta_asymptotic} For all $a\in A$ and for all $f\in W^s(M)$ with $s>2$ we have
$$
 \lim_{T\to +\infty}   \Vert  \frac{ 1} {T^{1/2}}  \int_0^{T}  f\circ \phi^X_t dt - \theta^f\left(g_{\log T}(a)\right)  \Vert_{L^2(M)}  =0\,.
$$
\end{theorem}

We leave to the reader the derivation of Lemma~\ref{lemma:theta_convergence}  and Theorem~\ref{thm:theta_asymptotic}. 

\smallskip
From Theorem~\ref{thm:theta_asymptotic} we can derive most of a result of Griffin and Marklof~\cite{GM} on limit distribution of theta sums
in the related context of Heisenberg nilflows. We also prove the result established later by Cellarosi and Marklof \cite{CM} (see in particular \cite{CM}, \S 3)  that limit 
distributions are the distributions of H\"older function of exponent equal to $1/2-$.  Our results have the advantage of holding  for all sufficiently smooth functions, while 
the work of Griffin and  Marklof~\cite{GM}, and Cellarosi and Marklof~\cite{CM} holds only for a single (toral) character. However,  they are much less explicit and less detailed, especially as far as the  the behavior at infinity in the moduli space is concerned,  and in particular we have no results on limit distributions for time sequences corresponding
to escape in the cusp of the moduli space.

\smallskip
The following result summarizes our results on limit distributions of ergodic averages of sufficiently smooth functions for Heisenberg nilflows:
\begin{theorem}  \label{thm:lim_dist} 
Let  $a=(X,Y,Z) \in A$  and let $(T_n)$ be any sequence such that 
$$
\lim_{n\to+\infty} g_{\log T_n}  (\bar a)  =  a_\infty \in \mathcal M \,.
$$
For every zero average function $f\in W^s(M)$ with $s>7/2$ which is not a coboundary,  the limit distribution 
 of the family of random variables 
$$
E_{T_n} (f) := \frac{1}{T_n^{1/2}}  \int_0^{T_n} f \circ \phi^X_t (\cdot) dt 
$$
exists and is equal to the distribution of the function $\theta^f (a_\infty)\in L^2(M)$. In particular, if $a_{\infty}\in DC$ belongs to  the $\omega$-limit set  of 
any $g_\R$-orbit  on the set  $DC$ of Diophantine points and is a continuity point of the Bufetov functional $\beta^f$ on $DC\subset A$, then $\theta^f (a_\infty)$ is  
almost everywhere equal to a bounded $\frac{1}{2}$-H\"older function on $M$, hence in particular the limit distribution has compact
support.
\end{theorem}
\begin{proof} Since $a_\infty \in \mathcal M$, the existence and characterization of the limits follows from Lemma~\ref{lemma:theta_convergence}  and Theorem~\ref{thm:theta_asymptotic}. The regularity statement follows from Corollary~\ref{cor:buf_theta}.  \end{proof}

With the exception of the H\"older regularity statement, equivalent results were proved in \cite{GM} for linear toral skew-shift and for function cohomologous to the
principal toral character. For such functions, the authors also investigated the case when the limit point $a_\infty =+\infty$ and proved that in that case
the limit distribution is the Dirac delta $\delta_0$ at $0\in \R$. The H\"older regularity property was proved in \cite{CM}.

\section{Square mean lower bounds} \label{sec:square_mean_bounds}

In this section we prove transverse square mean lower bound for ergodic integrals.

Let $\T^2_\Gamma$ denote the $2$-dimensional torus transverse to flow, defined as follows:
$$
\T^2_\Gamma := \{ \Gamma \exp( yY_0 +zZ) \vert  (y,z)\in \R^2\}.
$$
The torus $\T^2_\Gamma$ is transverse to the nilflow $\phi^{X_0}_\R$ on $M$, hence transverse to all nilflows $\phi^{X}_\R$
such that $<X, X_0> \not =0$.  For all $a =(X,Y,Z)$, let 
$$
t_a :=  \frac{1} {\vert <X,X_0>\vert}
$$
denote the return time of the flow $\phi^{X}_\R$ to the transverse tori $\T^2_\Gamma$. 

We will prove bounds for the square mean of ergodic integrals along the leaves of the foliation  of the 
torus $\T^2_\Gamma$ into circles transverse to the central direction:
$$
\{  \xi \exp (y Y_0)  \vert  y\in \T\}_{\xi \in \T^2_\Gamma} \,.
$$
\begin{lemma} \label{L2growth_circles} There exists a constant $C>0$ such that  for all $a=(X,Y,Z)$  and for every
irreducible component $H$ of central parameter $n\not =0$, there exist a function $f_H \in C^\infty(H)$  such that
$$
\begin{aligned}
&\vert f_H \vert_{L^\infty(M)} \leq C t_a^{-1}  \vert D^H_a(f_H) \vert  \,,   \\  &\vert f_H \vert_{a,s} \leq   C  t_a^{-1} \vert D^H_a(f_H) \vert (1+ t_a^{-1} \Vert Y\Vert)^{s}  (1+n^2)^{s/2} \,,
\end{aligned} 
$$
and such that, for all $x\in \T^2_\Gamma$ and $T \in \Z t_a$, we have
$$
\begin{aligned}
&\vert \int_\T \left(\int_0^Tf_H(\phi_s^X(\phi^{Y_0}_y(x)))ds\right) dy \vert   \leq 1 \,,  \\
&\|\int_0^Tf_H(\phi_s^X(\phi^{Y_0}_y(x)))ds\|_{L^2(\T, dy)} =   \vert D^H_a(f_H) \vert  (\frac{T}{t_a})^{1/2}\,.
\end{aligned}
$$
In addition, whenever $H \perp H' \subset L^2(M)$ the functions
$$
\int_0^Tf_{H}(\phi_s^X(\phi^{Y_0}_y(x)))ds\quad \text { and } \quad \int_0^Tf_{H'}(\phi_s^X(\phi^{Y_0}_y(x)))ds
$$
are orthogonal in $L^2(\T, dy)$. 
\end{lemma}
\begin{proof}  We recall that whenever $<X,X_0>\not =0$ the return map of the flow $\phi^X_\R$ to the transverse torus $\T^2_\Gamma$ is a linear skew-shift, that is, 
a map of the form
$$
T_{\rho,\sigma} (y, z) =  (y +\rho, z+  y + \sigma)  \quad \text{  on }  \R/\Z \times \R/K^{-1}\Z\,,
$$
for constants $\rho$, $\sigma\in \R$.    The operator $ I_a: L^2(M) \to L^2(\T^2_\Gamma)$  defined as
$$
f \to  I_a(f):=\int_0^{t_a}  f\circ \phi^X_t (\cdot) dt  \,, \quad  \text {for all } f\in L^2(M)
$$
is a surjective linear map of $L^2(M)$ onto $L^2(\T^2_\Gamma)$ with right inverse $R^\chi_a$ defined as follows: let $\chi \in C^\infty_0(0, 1)$ denote any function 
of integral equal to one and, for any $F\in L^2(\T^2_\Gamma)$,  let $R^\chi_a(F)\in L^2(M)$ be the function uniquely defined by the identity 
$$
R^\chi_a(F)(\phi^X_t(x))  =   t_a^{-1}\chi (t/t_a) F(x)  \,, \quad \text{ for all }  (x,t) \in \T^2_\Gamma \times [0, t_a]  \,.
$$
It is immediate from the definition that there exists a constant $C_\chi>0$ such that
$$
\vert  R^\chi_a(F) \vert_{a,s}  \leq  C_\chi  t_a^{-1} (1+ t_a^{-1} \Vert Y\Vert)^s  \Vert  F \Vert_{W^s(\T_\Gamma^2)}\,.
$$
As explained in \cite{AFU}, \S 5, the space $L^2(\T_\Gamma^2)$ can be decomposed as a direct sum of  irreducible subspaces invariant under the action of the map 
$T_{\rho,\sigma}$ on $L^2(\T_\Gamma^2)$ by composition. The subspace of functions with non-zero central character can be split as a direct sum of components 
$H_{(m,n)}$ with $(m,n) \in \Z_{K\vert n\vert} \times \Z \setminus \{0\}$. These spaces are defined as follows. Let $\{ e_{m,n} \vert (m,n)\in \Z^2\}$ denote the basis 
of characters of  $\T^2_\Gamma$,  that is the basis given by the functions
$$
e_{m,n} (y,z) := \exp [2\pi \imath (my+nKz) ]\,, \quad \text{ for all } (y,z) \in \T^2_\Gamma\,.
$$
As described in \cite{AFU},  \S 5, the functions $F \in H_{m,n}$ are characterized by a Fourier expansion of the form
$$
F= \sum_{j\in \Z}  F_j  e_{m+jn,n} \,.
$$
A generator of the space of invariant distributions is the distribution $D_{(m,n)}$ such that
$$
D_{(m,n)} ( e_{m+jn, n}) = e^{ -2\pi \imath [(\rho m+\sigma K n)j+\rho K n \binom{ j }{2} ] } \,.
$$
It follows that $\vert D_{(m,n)} (e_{m+jn,n}) \vert =1$. For any irreducible component $H$ of central parameter $n\not =0$, there exists 
$m\in \Z_{\vert n\vert}$ such that the operator $I_a:L^2(M) \to L^2(\T^2_\Gamma)$ maps the space $H$ onto the space $H_{m,n}$, hence
the function $f_H:= R^\chi_a(e_{m,n})\in C^\infty(H)$  has the property that such that $\vert D^H_a (f_H) = \vert D_{(m,n)} (e_{m,n})\vert =1$,  and
$$
\int_0^{t_a}   f_H\circ \phi^X_t (x)dt  =   e_{m,n} (x) \,, \quad \text{ for all } x \in \T^2_\Gamma\,.
$$
In addition the following estimates hold:
$$
\vert f_H \vert_{L^\infty(M)} \leq C_\chi t_a^{-1} \quad   \text{ and } \quad \vert f_H \vert_{a,s} \leq  C_\chi  t_a^{-1} (1+ t_a^{-1} \Vert Y\Vert)^{s} (1+K^2n^2)^{s/2}\,.
$$
Since for every $n\in  \Z\setminus \{0\}$ and $m\in \Z_{K\vert n\vert}$, the system 
$$
\{  \exp [2\pi \imath ( m(y  +j \rho) +   Kn (z + j(\sigma+ y) + \rho \binom {j}{2} ] \}_{j\in \Z} \subset L^2(\T, dy)
$$
is orthonormal, we have the identity
$$
\|\sum_{j=0}^{J-1}  \exp [2\pi \imath ( m(y +j\rho) +   Kn (z + j(\sigma+ y) + \rho \binom {j}{2} )   ] \|^2_{L^2(\T, dy)}  = J\,.
$$
In addition, by an immediate computation
$$
\int_\T \left( \sum_{j=0}^{J-1}  \exp [2\pi \imath ( m(y +j \rho) +   Kn (z + j(\sigma+ y) + \rho \binom {j}{2} )  ] \right) dy \in \{0, 1\}\,.
$$

By the above formula it follows that, whenever $T/t_a \in \Z $ we have
$$
\begin{aligned}
&\vert \int_\T\left(  \int_0^T  f_H \circ \phi^X_t ( \phi^{Y_0}_y (x) )dt \right)dy\vert 
=\vert\int_\T (\sum_{k=0}^{[\frac{T}{t_a}]-1} e_{m,n} \circ T_{\rho,\sigma}^k (y,z) )dy\vert \leq 1\,,\\
&\| \int_0^T  f_H \circ \phi^X_t ( \phi^{Y_0}_y (x) )dt \|_{L^2(\T, dy)} 
= \|\sum_{j=0}^{[\frac{T}{t_a}]-1} e_{m,n} \circ T_{\rho,\sigma}^j  \|_{L^2(\T, dy)}  = \left(\frac{T}{t_a}\right)^{1/2}\, ,
\end{aligned}
$$
The argument is therefore complete. 
\end{proof}

For any infinite dimensional vector ${\bf c}:= (c_{i,n})\in  \ell^2$,  let $\beta_{\bf c} $ denote the Bufetov functional defined as follows
\begin{equation}\label{eq:bc}
\beta_{\bf c} = \sum_{n\in \Z\not=0} \sum_{i=1}^{\mu(n)}    c_{i,n} \beta^{i,n} \,,.
\end{equation}
It follows from the orthogonality property and from Corollary~\ref{cor:L2bounds_buf} that the function $\beta_{\bf c} (a,\cdot, T)\in L^2(M)$ for all $(a,T)\in A \times \R^+$. In fact,
$$
\Vert \beta_{\bf c}(a, \cdot, T) \Vert^2_{L^2(M)} =   \sum_{n\in \Z\not=0}  \sum_{i=1}^{\mu(n)} \vert c_{i,n}\vert^2  \Vert \beta^{i,n}(a, \cdot, T) \Vert_{L^2(M)}^2 
\leq  C^2 \vert {\bf c}\vert^2_{\ell^2}  T \,.
$$
For any  vector ${\bf c}:= (c_{i,n})\in  \ell^2$, let  $\vert {\bf c} \vert_s$ denote the norm defined as
$$
\vert {\bf c} \vert^2_s  =  \sum_{n\in \Z\not=0} \sum_{i=1}^{\mu(n)}  (1+K^2n^2)^s \vert c_{i,n}\vert^2 \,.
$$
For any $a=(X,Y,Z) \in A$ such that  $<X, X_0> \not =0$ or, equivalently, such that the return time $t_a>0$ is finite, and for any $x\in M$
let $\mathcal S_{a,x}$ denote the transverse cylinder defined as follows:
$$
\mathcal S_{a,x} = \{ x \exp{ (y' Y + z' Z)} \vert    (y',z')  \in [0, t_a^{-1}) \times \T \}  \subset M\,.
$$
Let $\Phi_{a,x} : \T^2_\Gamma \to \mathcal S_{a,x}$ denote the maps defined as follows.  For any $\xi \in \T^2_\Gamma$, let $\xi' \in  \mathcal S_{a,x}$ 
denote the first intersection of the orbit $\{\phi^X_t (\xi)\vert t\geq 0\}$ with the transverse cylinder  ${\mathcal S}_{a,x}$. By definition there exists $t(\xi) \geq 0$  such that
\begin{equation}
\label{eq:Phi_ax}
\xi' =  \Phi_{a,x} (\xi) = \phi^X_{t(\xi)}  (\xi) \,, \quad \text{ for all } \xi \in \T^2_\Gamma\,.
\end{equation}
 Let  $(y,z)$ and $(y',z')$ denote the coordinates, respectively on $\T^2_\Gamma$ and ${\mathcal S}_{a,x}$, given by the 
exponential map, as follows
$$
(y,z) \to  \xi_{y,z} := \Gamma \exp (yY_0 + z Z) \in \T^2_\Gamma   \, \text{ and } \,   (y',z') \to  \xi' _{y',z'} := x \exp (y'Y + z' Z)\,.
$$
Let $X= \alpha X_0 + \beta Y_0 + v Z$ and $Y = \gamma X_0 + \delta Y_0 + wZ$ 
with $\alpha \not=0$ and $\alpha \delta- \beta \gamma =1$. Let $x = \Gamma  \exp (y_x Y_0 + z_x Z) \exp (t_x X_0)$ with
$(y_x, z_x)\in\ \T\times \R /K\Z$ and $t_x \in [0, 1)$. By the Baker-Campbell-Hausdorff formula, we derive that 
\begin{equation}
\label{eq:t_xi}
 \vert t(\xi ) \vert = \vert \delta t_x + \gamma (y-y_x) \vert   \leq   \Vert  Y \Vert \,. 
\end{equation}
and that the map $\Phi_{a,x} : \T^2_\Gamma \to {\mathcal S}_{a,x}$  is given by formulas of the following form: there exists a polynomial $P(a,x,y)$ of total degree $4$, quadratic with respect to each of the variables $(a,x,y) \in {\mathcal M}\times M\times \R$, such that
\begin{equation}
\label{eq:Phi_map}
\Phi_{a,x}(y,z)=:  \quad\begin{cases} y' = \alpha (y-y_x) + \beta t_x  \,,  \\
z' = z +  P(a,x,y) \,. \end{cases}
\end{equation}
In particular, the map $\Phi_{a,x}$ is invertible and such that 
$$
\Phi_{a,x}^* (dy' \wedge dz' )  =    \alpha \,dy \wedge dz = t_a^{-1} dy \wedge dz\,.
$$
The cylinders $\mathcal S_{a,x}$ are foliated by images of the circles
$ \{\xi \exp (y Y_0) \vert y\in \T \} \subset \T^2_\Gamma$ under the map $\Phi_{a,x}$. 

\begin{lemma} \label{lemma:L2growth_curves}
For any $s>7/2$ there exists a constant $C_s>0$ such that, for all Diophantine $a \in DC(L)$, for all ${\bf c} \in \ell^2$,
for all $z\in \T$ and all $T>0$, we have
$$
\begin{aligned}
&\vert   \int_\T \beta_{\bf c}(a, \Phi_{a,x} (\xi_{y,z}) , T)  dy \vert   \leq C_s (t_a+t_a^{-1}) (1+ L) (1+ t_a^{-1} \Vert Y\Vert)^{s} \vert {\bf c}\vert_s  \,,    \\
&\left\vert  \|  \beta_{\bf c}(a, \Phi_{a,x} (\xi_{y,z}) , T)      \|_{L^2(\T, dy)} -   \left(\frac{T}{t_a}\right)^{1/2}  \vert {\bf c} \vert_0 \right\vert  \leq 
C_s (t_a+t_a^{-1})(1+ L) (1+ t_a^{-1} \Vert Y\Vert)^{s} \vert {\bf c}\vert_s\,.
\end{aligned}
$$
\end{lemma}

\begin{proof} By Lemma~\ref{L2growth_circles}, for every $n\not =0$ and every $i\in \{1, \dots, \mu(n)\}$,  there exists a function $f_{i,n}\in C^\infty(H_{i,n})$ with
$D^{i,n}(f_{i,n}) =1$  such that, for all $T\in \Z t_a$ we have, for all $\xi \in \T^2_\Gamma$ the identities
$$
\begin{aligned}
& \vert \int_\T \left( \int_0^T f_{i,n} (\phi_s^X(\phi^{Y_0}_y(\xi)))ds \right)  dy \vert  \in \{0,1\} \,, \\
&\|\int_0^Tf_{i,n} (\phi_s^X(\phi^{Y_0}_y(\xi)))ds\|_{L^2(\T, dy)} =   (\frac{T}{t_a})^{1/2}\,.
\end{aligned}
$$
In addition, the integrals
$$
\int_0^Tf_{i,n} (\phi_s^X(\phi^{Y_0}_y(\xi)))ds
$$
form an orthogonal system of functions in $L^2(\T, dy)$.   Let then 
\begin{equation}
\label{eq:f_c}
f_{\bf c} :=  \sum_{n\in \Z\not=0} \sum_{i=1}^{\mu(n)} c_{i,n} f^{i,n} \,,
\end{equation}
By construction we have $\beta_{\bf c}= \beta^{f_{\bf c}}$. It is immediate that
$$
 \vert \int_\T \left( \int_0^T f_{\bf c} (\phi_s^X(\phi^{Y_0}_y(\xi)))ds \right)  dy \vert  \leq  \vert {\bf c} \vert_{\ell^1}\,.
$$
By orthogonality we also have
$$
\|\int_0^T f_{\bf c} (\phi_s^X(\phi^{Y_0}_y(\xi)))ds\|_{L^2(\T, dy)} =  (\frac{T}{t_a})^{1/2} \vert {\bf c} \vert_0\,.
$$
From the estimates on the functions $f_{i,n}$ stated in Lemma~\ref{L2growth_circles} we  derive the bounds
$$
\vert f_{\bf c} \vert_{L^{\infty}(M)} \leq C \vert {\bf c}\vert_{\ell^1} \quad \text{ and }\quad\vert f_{\bf c} \vert_{a,s} \leq  Ct_a^{-1} (1 + t_a^{-1} \Vert Y\Vert)^s  
 \vert {\bf c} \vert_s\,.
$$
From this estimate it follows that, for every $z \in \T$  and for all $T>0$,  we have
$$
\Vert \int_0^T f_{\bf c} (\phi_s^X\circ\Phi_{a,x}(\xi_{y,z}))ds - \int_0^T f_{\bf c} (\phi_s^X(\xi_{y,z}))ds \Vert_{L^2(\T,dy)}    \leq  2 \vert f_{\bf c} \vert_{L^\infty(M)}  \Vert Y \Vert \,.
$$
Finally let $T_a:= t_a( [T/t_a]+1) \in \Z t_a$. We have 
$$
\Vert \int_0^T f_{\bf c} (\phi_s^X(\xi_{y,z}))ds - \int_0^{T_a} f_{\bf c} (\phi_s^X(\xi_{y,z}))ds \Vert_{L^2(\T,dy)}    \leq   t_a \vert f_{\bf c} \vert_{L^\infty(M)} \,.
$$
We have therefore derived  that, for some constant $C'>0$ and for all $T>0$, the following bounds hold:
$$
\begin{aligned}
&\vert  \int_\T \left( \int_0^T f_{\bf c} (\phi_s^X(\Phi_{a,x} (\xi_{y,z}))ds \right) dy \vert \leq C' t_a (1 + t_a^{-1}\Vert Y \Vert) 
\vert {\bf c} \vert_{\ell^1}\,, \\
&\left\vert \|\int_0^T f_{\bf c} (\phi_s^X(\Phi_{a,x} (\xi_{y,z}))ds\|_{L^2(\T, dy)} -     \left(\frac{T}{t_a}\right)^{1/2}  \vert {\bf c} \vert_0  \right\vert \leq C' t_a (1 + t_a^{-1}\Vert Y \Vert) 
\vert {\bf c} \vert_{\ell^1}\,.
\end{aligned}
$$
By the asymptotic property of Theorem~\ref{thm:asymptotic}, for all $s>7/2$ there exists a constant $C_s>0$ such that  we have the uniform estimate
$$
\left\vert \int_0^{T}  f \circ \phi^X_t (x) dt - \beta^f (a, x, T)   \right\vert \leq  C_{s} (1+L) \vert f \vert_{a,s}\,.
$$
Since, by the above bounds on the function $f_{\bf c}$, there exists constant $C'_s>0$ such that 
$$
C' t_a (1 + t_a^{-1}\Vert Y \Vert) \vert {\bf c} \vert_{\ell^1} +
C_{s} t_a^{-1}(1+L) \vert f_{\bf c} \vert_{a,s} \leq    C'_s (t_a+t_a^{-1}) (1+ L) (1+ t_a^{-1} \Vert Y\Vert)^{s} \vert {\bf c}\vert_s\,, 
$$
we arrive at the estimates claimed in the statement.
\end{proof}

\section{Analyticity of the functionals} \label{sec:analyticity}

In this section we will prove that, for all $ a=(X,Y,Z)\in DC$ and for all $T\in \R$, the Bufetov functionals $\beta_H(a,\cdot, T)$ are real analytic  along the foliation tangent to
the integrable distribution $\{Y,Z\}$. This result is crucial in deriving measure estimates for the level sets of the Bufetov functionals and for our results on decay of correlations
of time changes.

By the orthogonality property, for every $T>0$, the Bufetov cocycle belongs to a single irreducible component $H$ (with central parameter $n \in \Z\setminus\{ 0\}$),
hence in particular (or from its definition), for all $(x,T)\in M\times \R$ and for all $t\in \R$, 
\begin{equation}
\label{eq:Zbeta}
\beta_H(a, \phi^Z_t (x), T) =  e^{2 \pi \imath K n t} \beta_H(a,x, T)\,.
\end{equation}
Let $\gamma:[0, T] \to M$ a $C^1$ (or piece-wise $C^1$ parametrized path). For every $t\in \R$ we define
$$
\gamma^Z_t (s) = \phi^Z_{ts} ( \gamma(s)) \, \quad \text{ for all }  s\in [0,T]\,.
$$
\begin{lemma} \label{lemma:central_twist}
The following formula holds:
$$
\hat \beta_H(a,\gamma^Z_t) =  e^{2\pi \imath t n K T} \hat \beta_H(a,\gamma) -2 \pi \imath n K t \int_0^T  e^{2\pi \imath t n K s} \hat \beta_H(a,\gamma\vert_{[0,s]}) ds \,.
$$
\end{lemma}
\begin{proof}   Let $a=(X,Y,Z)$ and let $\alpha$ be a $1$-form supported on a single irreducible component $H$.
As above we have the decomposition
$$
\alpha  = \alpha_X  \hat X + \alpha_Y \hat Y + \alpha_Z \hat Z\,.
$$

Let us compute the pairing of the current $\gamma^Z_t $ with the $1$-form $\alpha$ on $M$. By definition the tangent
vector of the path $\gamma^Z_t$ is given by the formula
$$
\frac{d\gamma^Z_t }{ds} =  D\phi^Z_{ts} ( \frac{d\gamma}{ds}) +   t Z \circ \gamma^Z_t \,.
$$
It follows that the pairing is given by the formula
$$
<\gamma^Z_t, \alpha> = \int_0^T  [\alpha (D\phi^Z_{ts} ( \frac{d\gamma}{ds})(s)) +  \imath_Z\alpha \circ \gamma^Z_t(s)] \, ds\,.
$$
Since $Z$ belongs to the center of the Lie algebra and the coefficients $\alpha_X$, $\alpha_Y$ and $\alpha_Z$ of $\alpha$ are eigenfunctions of the subgroup generated by $Z$ of eigenvalue $2\pi \imath n K \in 2\pi \imath K \Z$, it follows that
$$
<\gamma^Z_t, \alpha> = \int_0^T  [ e^{2\pi \imath n K t s}  \alpha ( \frac{d\gamma}{ds}(s)) +  \imath_Z\alpha \circ \gamma^Z_t(s)] \, ds\,.$$
Integration by parts gives
$$
\begin{aligned}
\int_0^T  e^{2\pi \imath nK t s}  \alpha ( \frac{d\gamma}{ds}(s))ds &= e^{2\pi \imath nK t T}  \int_0^T   \alpha ( \frac{d\gamma}{ds}(s)) ds \\&- 2\pi \imath n Kt \int_0^T   e^{2\pi \imath nKt s}  [\int_0^s \alpha ( \frac{d\gamma}{dr}(r) )  dr ] ds\,,
\end{aligned}
$$
hence we have the formula
$$
\begin{aligned}
<\gamma^Z_t, \alpha> = e^{2\pi \imath n K t T} <\gamma, \alpha> - 2\pi \imath n Kt &\int_0^T   e^{2\pi \imath n Kt s}  <\gamma\vert_{[0,s]}, \alpha> ds  \\ &+
\int_0^T (\imath_Z\alpha \circ \gamma^Z_t)(s) ds\,.
\end{aligned}
$$

Since the flow $g_\R$ is identity on the center $Z$, it follows that 

$$
\lim_{t\to+\infty}  e^{-\frac{t}{2}} \int_0^T (\imath_Z (\rho_{-t})^*\alpha \circ \gamma^Z_t)(s) ds  =0\,,
$$
hence the stated formula follows by the definition of the Bufetov functional and by the linearity of the projection operators.

\end{proof}

\begin{lemma}  \label{lemma:Ybeta} The following formula holds:
$$
\beta_H(a,\phi^Y_t (x), T) = e^{-2\pi \imath t n K T} \beta_H(a,x, T)  
+2 \pi \imath n K t \int_0^T  e^{-2\pi \imath  t n K s}  \beta_H(a,x, s) ds \,.
$$
\end{lemma}
\begin{proof}   We have the following commutation identities:
$$
\begin{aligned}
x  \exp({sX}) \exp({tY}) 
 &= x \exp ({tY}) \exp({sX})  \exp({ts Z})\,.
\end{aligned}
$$
Let then $\gamma^X_T (s) := \phi^X_s (x)$ for all $s\in [0, T]$. By definition the symbol $[\gamma^X_T]^Z_t$  denotes
the path given by the formula
$$
[\gamma^X_T]^Z_t (s) :=  \phi^Z_{ts} ( \gamma^X_T (s) ) \,, \quad \text{ for all } s \in [0, T]\,.
$$
It then follows by the definitions that 
$$
\phi^Y_t (\gamma^X_T(x))  =  [  \gamma^X_T (\phi^Y_t(x)) ]^Z_{t} \,.
$$
By the invariance property of the Bufetov functional and by Lemma \ref{lemma:central_twist} we have
$$
\begin{aligned}
\beta_H(a,x,T) &= \hat \beta_H(a,\phi^Y_t (\gamma^X_T(x)) ) =  e^{2\pi \imath t n KT} \hat \beta_H(a,\gamma^X_T (\phi^Y_t(x)))  \\ 
&-2 \pi \imath n Kt \int_0^T  e^{2\pi \imath  t n K s} \hat \beta_H(a,\gamma^X_T (\phi^Y_t(x))\vert_{[0,s]}) ds \\&=  e^{2\pi \imath t n KT} \beta_H(a,\phi^Y_t(x), T)  
-2 \pi \imath n Kt \int_0^T  e^{2\pi \imath  t n K s}  \beta_H(a,\phi^Y_t(x), s) ds
\end{aligned}
$$
The statement is an immediate consequence of the above formula.
\end{proof}

It follows from Lemma~\ref{lemma:Ybeta} and formula~\eqref{eq:Zbeta} that the Bufetov functional is {\it real analytic} (real and complex part are real analytic)
 on every leaf of the foliation tangent to the integrable distribution $\{Y,Z\}$.  
 
  For any $R>0$  let us introduce the analytic norm defined for all ${\bf c} \in \ell^2$ as
 $$
 \Vert  {\bf c} \Vert_{\omega, R} :=  \sum_{n\not =0} \sum_{i=1}^{\mu(n)}  e^{nR}  \vert c_{i,n} \vert 
 $$
 Let $\Omega_R$ denote the subspace of ${\bf c}\in \ell^2$ such that $ \Vert  {\bf c} \Vert_{\omega, R}$ is finite.
 
 \begin{lemma}  \label{lemma:buf_anal} For any ${\bf c} \in \Omega_R$,  any $a \in DC (L)$ and $T>0$, the functions defined as
 $$
 \beta_{\bf c} (a, \phi^Y_y \phi^Z_z (x),  T) \,,  \quad   \text{ for all } (y,z) \in \R \times \T\,,
 $$
 extends to a holomorphic function in the domain
 \begin{equation}
 \label{eq:D_RT}
 D_{R,T}:= \{ (y,z) \in \C\times \C/\Z \vert  \vert \text{\rm Im} (y) \vert T   +  \vert \text{\rm Im} (z) \vert   <   \frac{R}{2\pi K}\}\,.
 \end{equation}
 The following bounds hold: for any $R'<R$  there exists a constant $C_{R,R'}>0$ such that, for all $(y,z) \in D_{R',T}$ we have
 $$
\vert \beta_{\bf c} (a, \phi^Y_y \phi^Z_z (x),  T) \vert \leq C_{R,R'}  \Vert  {\bf c} \Vert_{\omega, R} \left(L+ T^{1/2} (1+ E_{\mathcal M}(a,T)) \right) (1+ K  \vert \text{\rm Im} (y) \vert   T ) \,.
 $$
 
 \end{lemma} 
 \begin{proof}
 By Lemma~ \ref{lemma:Ybeta} and formula~\eqref{eq:Zbeta}, for all $x\in M$   we have
 $$
 \begin{aligned}
 \beta_H (a, \phi^Y_y \phi^Z_z (x), T)&=  e^{2\pi \imath (z-yT) n K } \beta_H(a,x, T)  
\\ &+2 \pi \imath n K y e^{2\pi \imath z n K }  \int_0^T  e^{-2\pi \imath  y n K s}  \beta_H(a,x, s) ds \,.
\end{aligned}
$$
As a consequence, by Lemma~\ref{lemma:convergence} for $(y,z) \in \C \times \C/\Z$ we have 
$$
\begin{aligned}
\vert \beta_{\bf c} &(a, \phi^Y_y \phi^Z_z (x),  T) \vert \leq     C \left(L+ T^{1/2} (1+ E_{\mathcal M}(a,T)) \right) \sum_{n\not=0} \sum_{i=1}^{\mu(n)} \vert c_{i,n} \vert e^{2\pi \vert \text{\rm Im}(z-yT)\vert n K } \\
&+  2 \pi C \left(L+ T^{1/2} (1+ E_{\mathcal M}(a,T) \right) K  \vert \text{Im} (y) \vert   T   \sum_{n\not=0} \sum_{i=1}^{\mu(n)}  n \vert c_{i,n} \vert    e^{2\pi  (\vert \text{\rm Im} (y)\vert T + \vert \text{\rm Im} (z)\vert)  n K }  \,. 
\end{aligned}
 $$
  It follows that the function  $\beta_H (a, \phi^Y_y \phi^Z_z (x), T)$ is given by a series of holomorphic functions on $\C \times \C/\Z$, which converges uniformly
  on compact subsets of the domain $D_{R,T}$, hence it is holomorphic there. The stated uniform bound follows immediately from the proof.
 \end{proof}
 
 For every $\eta \in (0,1)$, let $\Omega^{(\eta)}_\infty$ denote the subset of $\Omega_\infty= \cap_{R>0} \Omega_R$ defined  as follows. A sequence  ${\bf c} \in
\Omega^{(\eta)}_\infty$ if  there exists $C_\eta>0$ such that, for all $R>0$ we have
\begin{equation}
\label{eq:Omega_eta}
 \sum_{n\not =0} \sum_{i=1}^{\mu(n)}  n \vert c_{i,n} \vert e^{nR}   \leq  C_\eta e^{R^{2-\eta}} \,.
\end{equation}
The set $\Omega^{(\eta)}_\infty$ is dense in $\ell^2$, since it contains all finitely supported sequences.

  \begin{lemma}  \label{lemma:buf_anal_bis}
  For any ${\bf c} \in \Omega^{(\eta)}_\infty$,  any $a \in DC (L)$ and $T>0$, the functions defined as
 $$
 \beta_{\bf c} (a, \phi^Y_y \phi^Z_z (x),  T) \,,  \quad   \text{ for all } (y,z) \in \R \times \T\,,
 $$
 extends to a holomorphic function in the domain $\C\times \C/\Z$.   In addition, there exists a constant $C_\eta>0$ such that, for all $T>0$ 
 and for all $(y,z) \in \C\times \C/\Z$, we have
 $$
 \begin{aligned}
\vert \beta_{\bf c} &(a, \phi^Y_y \phi^Z_z (x),  T) \vert \leq  C_\eta   \left(L+ T^{1/2} (1+ E_{\mathcal M}(a,T)) \right) \\  \times &
 (1+ 2\pi K \vert \text{\rm Im} (y) \vert   T)     \exp [(\vert \text{\rm Im}(y)\vert T + \vert \text{\rm Im}(z)\vert)^{2-\eta} ] \,.
\end{aligned} 
$$
 \end{lemma}
\begin{proof}  Since $\Omega^{(\eta)}_\infty \subset \Omega_R$, for all $R>0$, it follows already from Lemma~\ref{lemma:buf_anal} that the function 
$ \beta_{\bf c} (a, \phi^Y_y \phi^Z_z (x),  T)$ extends to a holomorphic function on $\C \times \C/\Z$. As in the proof of Lemma~\ref{lemma:buf_anal},
by Lemma~\ref{lemma:convergence} for $(y,z) \in \C \times \C/\Z$ we have 
$$
\begin{aligned}
\vert \beta_{\bf c} &(a, \phi^Y_y \phi^Z_z (x),  T) \vert \leq  C' \left(L+ T^{1/2} (1+ E_{\mathcal M}(a,T)) \right)    \\ &  \times
(1+ 2\pi K \vert \text{Im} (y) \vert   T)   \sum_{n\not=0} \sum_{i=1}^{\mu(n)}  n \vert c_{i,n} \vert    e^{2\pi K (\vert \text{\rm Im} (y)\vert T + \vert \text{\rm Im} (z)\vert)  n } \,.
\end{aligned}
$$
Since by assumption ${\bf c} \in \Omega^{(\eta)}_\infty$, the stated estimates is proved.
\end{proof}

In the Sections~\ref{sec.mes} and~\ref{sec.mesgen} we will use Lemmas \ref{lemma:buf_anal} and~\ref{lemma:buf_anal_bis} to get uniform measure estimates on 
sets where the Bufetov functional is small.  This is possible thanks to results on the measure of small sets for analytic functions (see \cite{BruGa}, \cite{Bru}).  

\section{Bounds on the valency}
\label{sec:valency}

For convenience of the reader we recall a result of A.~Brudnyi on the measure of level sets of analytic functions.

For any $r>1$, let $\mathcal O_r$ denote the space of holomorphic functions on the ball $B_\C (0,r) \subset \C ^n$.  Let $B_\R(0,1) := B_\C (0,1) \cap \R^n$
denote the real euclidean unit ball.

\begin{theorem} (\cite{Bru}, Thm. 1.9)\label{thm:bru} For any $f\in \mathcal O_r$ there is a constant $d:=d(f,r)>0$ such that for any convex set $D \subset B_\R(0,1)$, 
for any measurable subset $\omega  \subset   D$
$$
\sup_{D} \vert f\vert   \leq \left( \frac{ 4n \text{\rm Leb}(D) }  {\text{\rm Leb}(\omega) } \right)^d   \sup_{\omega}  \vert f\vert  \,.
$$

\end{theorem} 

The best constant $d$ in the above theorem  is called the {\it Chebyshev degree}, denoted by $d_f (r)$, of the function $f\in \mathcal O_r$ in $B_\C(0, 1)$.
The Chebyshev degree can be estimated by the valency of the function. We recall the definition of the valency.

A  holomorphic function $f$ defined in a disk is called $p$-valent if it assumes no value more than $p$-times there (counting multiplicities). We also say 
that $f$ is $0$-valent if it is a constant. For any $t \in [1, r)$,  let $\mathcal L_t$ denote the set of one-dimensional complex affine spaces $L\subset \C^n$ such 
that $L\cap B_\C(0, t)\not=\emptyset$.

\begin{definition} (\cite{Bru}, Def. 1.6) Let $f\in \mathcal O_r$.  The number
$$
v_f(t): =  \sup_{L\in \mathcal L_t}  \{ \text{\rm valency of $f \vert L\cap B_\C(0,t)$}    \}\,
$$
is called the \emph{valency} of $f$ in $B_\C(0,t)$.

\end{definition}

By \cite{Bru} Prop. 1.7, for any $f \in \mathcal O_r$ and any $t \in [1,r)$ the valency $v_f (t)$ is finite and there is a constant $c:=c(r)>0$ such that 
\begin{equation}
\label{eq:d_f}
d_f(r)\leq  c v_f(\frac{1+r}{2}) \,.
\end{equation}

In this section we prove the following result.
\begin{lemma} \label{lemma:valency_normal} Let $R >r>1$. For any normal family $\mathcal F \subset \mathcal O_R$ such that no functions in 
$\overline{\mathcal F} =\emptyset$ is constant along a one-dimensional complex line, we have
$$
\sup_{ f\in \mathcal F} v_f (r) < +\infty\,.
$$
\end{lemma}
\begin{proof} We argue by contradiction. If the statement does not hold, then there exists a sequence of functions $(f_k) \subset \mathcal F$,
a sequence of affine one-dimensional subspaces $(L_k)$  and a bounded sequence of complex numbers such that
$$
\# f_k^{-1} \{w_k\} \cap  L_k \cap B_\C (0, r)  \to + \infty \,.
$$
By compactness, up to passing to subsequence we can assume that $f_k \to f$ uniformly on all compact subset of the ball $B_\C(0,R)$,
that $L_k \to L$, a one-dimensional affine complex line such that $L \cap \overline {B_\C(0,r)} \not=\emptyset$, in the Hausdorff topology,
and that  $w_k \to w \in \C$. By hypothesis $f\vert L$ is non-constant, hence we can assume that $f_k\vert L_k$ is also non-constant for all $k\in \N$. 
Since for any $r'>r$ the valency  $v_f (r')$ of the function $f$ on $B_\C (0,r')$ is finite we have that 
$$
\# f^{-1} \{w\} \cap  L \cap B_\C (0, r') <+\infty\,.
$$
Let $f^{-1} \{w\} =\{p_1, \dots, p_v\} \subset L \cap B_\C (0, r')$. Let $\epsilon >0$  be chosen so that $B_\C(p_i, \epsilon) \cap 
B_\C(p_j, \epsilon) = \emptyset$ and $f\vert \partial B_\C(p_i, \epsilon) \not =0$ for all $i,j\in \{1, \dots, v\}$. Since $L_k \to L$
there exists a sequence of affine holomorphic maps $A_k :\C^n \to \C^n$ such that $A_k \to Id$ uniformly on compact sets and
$A_k (L)= L_k$ for all $k\in \N$.  By uniform convergence we have that for $n\in \N$ sufficiently large  all the solutions 
$z \in L\cap B_\C(0,r')$ of the equation $f_k\circ A_k (z) - w_k =0$ are contained in the union of the balls $B_\C(p_i, \epsilon)\cap L$. 
For all $k\in \N$, let $\phi_k := (f_k \circ A_k)\vert L $, and let $\phi= f\vert L$. The sequence of holomorphic functions $(\phi_k)$
converges to $\phi$ uniformly on compact sets of $L\cap B_\C (0, R)$. Since $L\approx \C$ by the residue theorem we have that
$$
\# (f_k \circ A_k)^{-1} (w_k) \cap B_\C (p_i, \epsilon) \cap L  =  \frac{1}{2\pi \imath} \int_{\partial B_\C (p_i, \epsilon)}  
 \frac{  \phi'(z)_k}{\phi_k(z) -w_k} dz \,.
$$
By uniform convergence on compact sets it follows that
$$
\frac{1}{2\pi \imath} \int_{\partial B_\C (p_i, \epsilon)}   \frac{  \phi'_k(z)}{\phi_k(z) -w_k} dz \to  
\frac{1}{2\pi \imath} \int_{\partial B_\C (p_i, \epsilon)}   \frac{  \phi'(z)}{\phi (z) -w} dz\,,
$$
hence we have
$$
\# (f_k \circ A_k)^{-1} (w_k) \cap B_\C (p_i, \epsilon) \cap L  \to   \# f^{-1} (w) \cap B_\C (p_i, \epsilon) \cap L \,.
$$
We conclude that for $k$ sufficiently large
$$
\# f_k^{-1} (w_k) \cap L_k \cap B_\C(0,r) \leq  \# f^{-1} (w) \cap L \cap B_\C(0,r') <+\infty\,.
$$
Since by assumption the LHS in the above inequality diverges, we have reached a contradiction. The argument is concluded.

\end{proof}
In one complex dimension we prove a quantitative bound on the valency.
\begin{lemma}
\label{lemma:valency_bound_onedim} For any $R>r>3t>3$, there exists a constant $C_{r,t}>0$ such that the following holds.
For any non-constant holomorphic function of one complex variable $f\in \mathcal O_R$, let $M_f(r)$ denote the maximum 
modulus of $f$ on the closed ball $\overline{B_\C (0,r)} \subset \C$ and let $O_f(t)$ its  oscillation in the ball $B_\C (0,t)$. 
 The valency $\nu_f(t)$ of the function $f$ in the ball $B_\C (0,t)$ satisfies the following estimate
$$
\nu_f(t) \leq   C_{r,t}  \log \left ( 4 \frac{M_f(r)}{ O_f(t)} \right)\,.
$$
\end{lemma}
\begin{proof}  Since there exists a single complex one-dimensional affine space $L\subset \C$, it suffices to estimate the valency
of the function $f$ on $B_\C(0,t)$, that is, the number of solutions $z\in B_\C(0,t)$ of equations of the form $f(z)=w$.  

By definition,
the above equation has solutions only if $\vert w \vert \leq M_f(r)$. Let $f_w\in  \mathcal O_R$ denote the holomorphic function
$f_w(z) = f(z)-w$. By definition, the maximum modulus of $f_w$ on the closed ball $B_\C (0,r) \subset \C$ is at most $2M_f(r)$. 
Let $w \in  f(B_\C (0,t)$ and let $z_w \in B_\C(0,t)$ be any point such that
$$
\vert f_w (z_w) \vert = \vert f_w(z) -w \vert  \geq  O_f(t) /2\,.
$$
Let $\{z_1, \dots, z_\nu\} \subset B_\C(z_w, 2t)\setminus \{z_w\}$ denote the (non-empty) set of zeros of the function 
$f_w$ in $B_\C(z_w,2t)$ listed with their multiplicities. Since $B_\C(0,t) \subset  B_\C (z_w, 2t)$ it follows that
the number of solution of the equation $f_w(z)=0$ in $B_\C(0,t)$ is at most $\nu \in \N$. 
Let us define
$$
g_w (z) =  f_w (z_w+z) \prod_{k=1}^\nu (1- \frac{z_w+z}{z_k} )^{-1}    \,, \quad  z\in B_\C (0,R-t)\,.
$$
By definition the function $g_w$ in holomorphic in $B_\C(0,R-t)$. By the maximum modulus principle
$$
\frac{1}{2} O_f(t)  \leq \vert f_w (z_w)\vert = \vert g_w (0)\vert  \leq \max_{\vert z\vert =r-t}  \vert g_w(z)\vert \leq  2 M_f(r)  (\frac{r-t}{3t} -1)^{-\nu} \,,
$$ 
which immediately implies, by taking logarithms,
$$
\nu   \leq    \log (\frac{r-t}{3t} -1) \,  \log  \left( \frac{ 4M_f(r)}{ O_f(t)} \right) \,.
$$
The statement is therefore proved.
\end{proof}

\section{Measure estimates:the bounded-type case} \label{sec.mes}

Finally, we derive a bound on the valency, hence on the Chebyshev degree of the holomorphic extensions of Diophantine Bufetov functionals,
uniform over compact invariant subset of the moduli space.

\begin{lemma}\label{lemma:compact1} Let  $L>0$ and let $\mathcal B\subset DC(L)$ be a bounded subset. Given $R>0$, for all ${\bf c}\in \Omega_R$  and  $T>0$,  let 
$\mathcal{F}({\bf c},T)$ denote the family of real analytic functions of the variable $y\in [0,1)$ defined as follows:
$$
\mathcal{F}({\bf c},T):=\{\beta_{\bf c}(a,\Phi_{a,x} (\xi_{y,z}), T)\,\vert ( a,x,z)  \in \mathcal B \times M \times \T\}\,.
$$
 There exist $T_{\mathcal B}>0$ and $\rho_{\mathcal B}>0$, such that for every  
$(R,T)$  such that  $R/T \geq \rho_\mathcal B$ and $T\geq T_{\mathcal B}$,  and for all ${\bf c}\in \Omega_R\setminus\{0\}$, we have
$$
 \sup_{f\in \mathcal{F}({\bf c},T)}v_{f}<+\infty.
$$
\end{lemma}
\begin{proof} Since  ${\mathcal B}\subset \mathcal M$ is bounded, we get 
$$
0\,<\, t_{\mathcal B}^{min}=  \inf_{a\in {\mathcal B} } t_{a}\leq \sup_{a\in {\mathcal B}} t_a = t_{\mathcal B}^{max}\,<\,+\infty.
$$
For any $ a\in {\mathcal B}$ and $x\in M$, the map $\Phi_{a,x}: [0,1) \times \T \to [0, t_{a}^{-1}) \times \T$ introduced in formula \eqref{eq:Phi_map} extends to 
a complex analytic diffeomorphism $\hat   \Phi_{a,x}: \C \times \C/\Z  \to \C \times \Z$.  By Lemma~\ref{lemma:buf_anal} It follows that the real analytic function 
$$
\beta_{\bf c}(a,\Phi_{a,x} (\xi_{y,z}), T)  \,,   \text{ for } (y,z) \in [0, 1) \times \T\,,
$$
extends to a holomorphic function on the domain $\hat   \Phi_{a,x}^{-1} (D_{R,T}) \subset \C \times \C/\Z$. 
In particular, for every $z\in \T$ the function 
$$
\beta_{\bf c}(a,\Phi_{a,x} (\xi_{y,z}), T)     \,,   \text{ for } y \in [0, 1)\,,
$$ 
extends to a holomorphic function defined on a strip 
$$
H_{a,x, R,T} :=\{ y\in \C \vert  \vert \text{ \rm Im} (y) \vert < h_{a,x,R,T}\}\,.
$$
Moreover by the boundedness of the set $\mathcal B \subset \mathcal M$ it follows that 
$$
 \inf_{(a,x) \in {\mathcal B}\times M} h_{a,x,R,T}:= h_{R,T} >0\,.
$$
 In fact, the function $h_{a,x,R,T}$ and its lower bound $h_{R,T}$ can be computed explicitly from the formula \eqref{eq:Phi_map}   for the polynomial map $\Phi_{a,x}$ 
 and from definition of the domain $D_{R,T}$ in formula~\eqref{eq:D_RT}.  In particular, it follows that for every $r>1$ there exists $\rho_\mathcal B\gg 1$ such that, for every
 $R$, $T$ such that $R/T \geq \rho_\mathcal B$, then for every $( a,x, z) \in \mathcal B \times M \times \T$ we have that, as a function of $y \in [0,1]$, 
 $$\beta_{\bf c}(a,\Phi_{a,x} (\xi_{y,z}), T)\in \mathcal{O}_r\,.$$
It then follows from Lemma \ref{lemma:buf_anal}  that the family $\mathcal{F}({\bf c},T)$ is uniformly bounded and hence normal. Moreover, from Lemma \ref{lemma:L2growth_curves} it follows that for sufficiently large $T>0$ no sequence from $\mathcal{F}({\bf c},T)$ can converge to a constant function. 
The statement finally follows from Lemma~\ref{lemma:valency_normal}. 
\end{proof}

We can finally derive crucial measure estimates on Bufetov functionals.

\begin{lemma}\label{lemma:compact} Let $a\in DC$ be such that the forward orbit $\{g_t(\bar{a})\}_{t\in\R^+}$ is contained in a compact subset of $\mathcal M$. There 
exist $R>0$,  $T_0>0$ and $C>0$, $\delta>0$ such that, for every ${\bf c}\in \Omega_R\setminus \{0\}$, $T\geq T_0$ and for every $\epsilon>0$, we have
\begin{equation}
\label{eq:orbint}
\text{\rm vol} (\{x \in M \vert  \vert \beta_{\bf c} (a, x, T)\vert \leq \epsilon T^{1/2} \} ) \leq  C   \epsilon^{\delta}\,.
\end{equation}
\end{lemma}
\begin{proof} Let $R>0$ and $T_0>0$ be chosen so that the conclusions of Lemma~\ref{lemma:compact1} hold and let ${\bf c} \in \Omega_R$. 
By the scaling property of Bufetov functionals
$$
\beta_{\bf c} (a, x, T) = (T/T_0)^{1/2}\beta_{\bf c} (g_{\log(T/T_0)} (a) , x , T_0)\,.
$$
Since $a\in DC$ and the $g_\R$-forward orbit $\{g_t(\bar{a})\}_{t\in\R^+}$ is contained in a compact set, there exists $L>0$ such that 
$g_t(a) \in DC(L)$ for all $t>0$. Since  the volume on $M$ is invariant under the action $A_\Gamma$, 
it is enough to estimate (uniformly over $(a,x) \in \mathcal B \times M$ for any givenbounded subset $\mathcal B \subset DC(L)$ such that 
$ \{g_t({a})\}_{t\in\R^+} \subset  A_\Gamma \backslash {\mathcal B}$),  for any $\epsilon>0$, the volume 
$\text{\rm vol} (\{x \in M \vert  \vert \beta_{{\bf c}}(a,x,T_0)\vert \leq \epsilon \})$.
By Fubini's theorem it is enough to estimate, uniformly over $(a,x,z)\in \mathcal B\times M\times \T$, 
$$
\text{\rm Leb} (\{y \in [0,1] \vert  \vert \beta_{{\bf c}}(a,\Phi_{a,x}(\xi_{y,z}),T_0)\vert \leq \epsilon \}).
$$
Let $\delta^{-1} := c(r) \sup_{f\in \mathcal{F}({\bf c},T_0)}v_f(\frac{1+r}{2})<+\infty$. Since by Lemma \ref{lemma:L2growth_curves} we have
$$
\inf_{(a,x,z) \in \mathcal B \times M\times \T}  \sup_{y\in[0,1]}\vert\beta_{{\bf c}}(a,\Phi_{a,x}(\xi_{y,z}),T_0)\vert>0\,.
$$
it follows from Theorem~\ref{thm:bru} for $D=B_{\mathbb{R}}(0,1)$ and 
$$
\omega:=\{y \in [0,1] \vert  \vert \beta_{{\bf c}}(a,\Phi_{a,x}( \xi_{y,z}),T_0)\vert \leq \epsilon\}\,,
$$ 
and by the bound in formula~\eqref{eq:d_f} for the Chebychev degree, that the following estimate holds: there exists a constant $C>0$ such that, 
for all $(a,x,z) \in \mathcal B \times M\times \T$ and for all $\epsilon >0$, we have
$$
\text{\rm Leb} (\{y \in [0,1] \vert  \vert \beta_{{\bf c}}(a,\Phi_{a,x}(\xi_{y,z}),T_0)\vert \leq \epsilon )\leq C \epsilon^\delta.
$$
The statement then follows by the Fubini theorem.
\end{proof}

\begin{corollary}\label{cor:compact} Let $a=(X,Y,Z_0)$ be as in Lemma \ref{lemma:compact}. There exist $R>0$, 
$T_0>0$ and $C>0$, $\delta>0$ such that, for every ${\bf c}\in \Omega_R\setminus \{0\}$, $T\geq T_0$ and for every $\epsilon>0$, we have
$$
\text{\rm vol} (\{x \in M \vert  \vert \int_0^Tf_{\bf c}(\phi_t^Xx)dt\vert \leq \epsilon T^{1/2} \} ) \leq  C\epsilon^\delta.
$$
\end{corollary}

\section{Measure estimates: the general case}
\label{sec.mesgen}
Bufetov functionals were constructed for $a=(X,Y,Z) \in A$ under a (full measure) Diophantine condition (DC) on $\bar a\in \mathcal M$, which depends on the backward 
 orbit under the renormalization flow $g_\R$ in the moduli space $\mathcal M$. Throughout this section we assume that $a\in DC$ satisfies  an additional  (full measure) Diophantine condition $DC_{log}$  (which depends on the forward orbit):   $a\in DC_{log}$ if $\bar a \in {\mathcal M}$ satisfies the logarithmic law of geodesics, that is, if 
\begin{equation}\label{dist.cusp}
\limsup_{t\to+\infty}\frac{\delta_{\mathcal{M}}(g_t(\bar{a}))}{\log t}=1.
\end{equation}

\begin{lemma}\label{lemma:gencase} Let $a\in DC\cap DC_{log}$. Let $\eta \in (0,1)$ and let ${\bf c} \in \Omega_\infty^{(\eta)}$. For every $\delta \in (1-\eta/2, 1)$ and for 
every $\zeta>0$,  there exist constants $C_{\delta, \zeta}>0$ and $C_\zeta>0$  such that,  for every $\epsilon>0$, we have
\begin{equation}
\label{eq:orbint_gen}
\text{\rm vol} (\{x \in M \vert  \vert \beta_{\bf c} (a, x, T)\vert \leq \epsilon \frac{T^{1/2}}{C_\zeta \log^{1/4+\zeta} T}  \} ) \leq   4 \epsilon^{\frac{1}{C_{\delta, \zeta} \log^\delta T}}\,.
\end{equation}
\end{lemma}

\begin{proof}
Let us assume $a \in DC_{log} \cap DC(L)$. By the Diophantine condition $DC_{log}$, there exists a bounded set ${\mathcal B} \subset {\mathcal M}$ such that the following holds.  For any $\zeta >0$ there exist a constant  $C_\zeta>0$ and a sequence $(t_n)$ with $g_{t_n} ({\bar a}) \in \mathcal B$ such that
\begin{equation}
\label{eq:t_n}
e^{t_{n+1} -t_n}  \leq  C_\zeta  t_n^{1+ \zeta}\,.
\end{equation}
Let $ T \gg 1$ and for all $n \in \N$ let $T_n = e^{-t_n} T$.   By the Diophantine condition $DC_{log}$,  for any $\zeta >0$ there exists a constant 
$C'_\zeta>0$ such that $g_{t_n} (a) \in DC(L_n)$  with
\begin{equation}
\label{eq:L_n}
\begin{aligned}
L_n =& e^{-t_n/2} L + {\mathcal E}_{\mathcal M}(a, e^{t_n}) \leq  e^{-t_n/2} L   +  C'_\zeta \, t_n^{ \frac{1}{4} +\zeta}  \\ =&  (T/T_n)^{-1/2} L + C'_\zeta \log^{\frac{1}{4} +\zeta} (T/T_n)  \,.
\end{aligned}
\end{equation}
By the scaling property of the Bufetov functionals
\begin{equation}
\label{eq:scale_n}
\beta_{\bf c} (a, x, T) = (T/T_n)^{1/2}\beta_{\bf c} (g_{t_n} (a) , x , T_n)\,.
\end{equation}
Since $\mathcal B$ is a bounded, hence relatively compact set, we have 
$$
0\,<\, t_{\mathcal B}^{min}=  \inf_{a\in {\mathcal B} } t_{a}\leq \sup_{a\in {\mathcal B}} t_a = t_{\mathcal B}^{max}\,<\,+\infty.
$$
By Lemma~\ref{lemma:buf_anal_bis} the real analytic function $$\beta_{\bf c}(g_{t_n} (a) , \phi^Y_y \phi^Z_z(x) , T_n) \,,  \quad \text{for all } (y,z) \in \R \times \R/\Z    $$   
extends to a complex analytic function on the strip $D^{(n)}_R \subset \C^2$ defined as
$$
D^{(n)}_R= \{ (y,z) \in \C \times \C/\Z \vert   \vert \text{Im}(y) \vert T_n +  \vert \text{Im}(z) \vert  \leq  \frac{R T_n}{2\pi K} \}\,,
$$ 
and, for any ${\bf c} \in \Omega^{(\eta)}_\infty$ there exists $C_\eta>0$ such that  for any $(y,z) \in D^{(n)}_R$ we have the uniform upper bound
$$
\vert \beta_{\bf c} (g_{t_n}(a), \phi^Y_y \phi^Z_z (x),  T_n) \vert \leq  C_\eta  
\left(L_n+ T_n^{1/2} (1+ E_{\mathcal M}(g_{t_n}(a),T_n)) \right) (1+ RT_n) e^{(RT_n)^{2-\eta} }
\,.
 $$
In particular it follows that,  for any $z\in \T$, the function $ \beta_{\bf c}(g_{t_n} (a) , \Phi_{g_{t_n} (a),x}(\xi_{y,z}) , T_n)$, defined for $y\in [0,1]$ extends to a holomorphic
function on the strip
$$
H_R = \{ y \in \C \vert  \,   \vert  \text{Im}(y)\vert  \leq   \frac{R t^{min}_{\mathcal B}}{4\pi K } \}\,,
$$
and  there exists a constant $C'_\eta>0$ such that  the following 
uniform upper bound holds: for any $y \in H_{R}$ and any $z\in \T$, we have 
\begin{equation}
\label{eq:upper_bound_strip}
\begin{aligned}
&\vert  \beta_{\bf c}(g_{t_n} (a) , \Phi_{g_{t_n} (a),x}(\xi_{y,z}) , T_n) \vert  \\  &\leq  C'_\eta  \left(L_n+ T_n^{1/2} 
(1+ E_{\mathcal M}(g_{t_n}(a),T_n)) \right) (1+ RT_n) e^{(RT_n)^{2-\eta} }\,.
\end{aligned}
\end{equation}
By a calculation, for all $ t_n\geq 0$ and for $T_n = e^{-t_n} T  \in [0, T]$ we have that
$$
E_{\mathcal M} (g_{t_n}(a),T_n)\leq E_{\mathcal M} (a,T) \leq    C_\zeta  \log^{1/4 +\zeta} T\,.
$$
By Lemma~\ref{lemma:L2growth_curves} it follows that, for any $s>7/2$,  whenever we have
\begin{equation}
\label{eq:T_n}
\left(\frac{T_n}{ t_{\mathcal B}^{max}}\right)^{1/2}  \vert {\bf c} \vert_0  \geq 10 
C_s (1+ L_n )  \vert {\bf c}\vert_s  \sup_{a\in \mathcal B} \{   t_a^{-1} (1+ t_a^{-1} \Vert Y \Vert)^{s+1} \}  \,,
\end{equation}
then there exists a constant $C_{\mathcal B}>0$ such that 
\begin{equation}
\label{eq:lower_bound_n}
\begin{aligned}
&\Vert \beta_{\bf c}(g_{t_n} (a) , \Phi_{g_{t_n} (a),x}(\xi_{y,z}) , T_n)  \Vert_{L^2(\T, dy)}  \geq   C_{\mathcal B}  \vert {\bf c} \vert_0  T_n^{1/2} \,;  \\ 
&\vert \int_\T  \beta_{\bf c}(g_{t_n} (a) , \Phi_{g_{t_n} (a),x}(\xi_{y,z}) , T_n)  dy \vert \leq  \frac{C_{\mathcal B}}{4}   \vert {\bf c} \vert_0  T_n^{1/2}     \,.
\end{aligned}
\end{equation}
In particular, we derive a uniform lower bound for the oscillation $O_n({\bf c}, T)$ of the function $\beta_{\bf c}(g_{t_n} (a) , \Phi_{g_{t_n} (a),x}(\xi_{y,z}) , T_n)$
for $y\in [0,1]$:
\begin{equation}
\label{eq:osc_n}
O_n({\bf c}, T) \geq   \frac{C_{\mathcal B}}{2}   \vert {\bf c} \vert_0  T_n^{1/2} \,.
\end{equation}
It remains to optimize the choice of $t_n>0$, hence of $T_n\in [0,T)$,  given $T>0$.   It follows from formulas \eqref{eq:L_n} and~\eqref{eq:T_n} that for any
$\zeta >0$, there exists a constants $L_\zeta>0$  such that we want to choose $T_n$ to be the smallest solution of the inequality
$$
T_n \geq  L^2_\zeta  \left (1+ \log^{\frac{1}{4}+\zeta}(T/T_n) \right)^2\,.
$$
By this definition and by the condition in formula~\eqref{eq:t_n} we then have 
$$
e^{t_n} \leq L_\zeta e^{t_n} (1+ t_n^{\frac{1}{4}+\zeta})^2 \leq T   <    L_\zeta e^{t_{n+1}} (1+ t_{n+1}^{\frac{1}{4}+\zeta})^2 \leq   L_\zeta C''_\zeta e^{t_n} (1+ t_n^{\frac{1}{4}+\zeta})^2 \,,
$$
which in turn implies
$$
 \frac{T}{T_n} =e^{t_n}  \geq   (L C''_\zeta)^{-1}  \frac{ T}  {(1+ \log^{\frac{1}{4} +\zeta} T)^2} \,.
$$
It follows in particular that 
\begin{equation}
\label{eq:T_n_est}
T_n \leq  (L C''_\zeta) (1+ \log^{\frac{1}{4} +\zeta} T)^2 \,, \quad L_n \leq L +  C'_\zeta L_\zeta^{-1} T_n^{1/2} \,.
\end{equation}
With this choice there exists a constant $C_{\eta, \zeta}>0$ such that 
\begin{equation}
\label{eq:upper_bound_strip_bis}
\begin{aligned}
\vert  \beta_{\bf c}(g_{t_n} (a) , &\Phi_{g_{t_n} (a),x}(\xi_{y,z}) , T_n) \vert  \leq  C_{\eta,\zeta}  \left(1+ T_n^{1/2} \log^{1/4+ \zeta}  T \right)  \\ \times &
(1+ R \log^{1/2+ 2\zeta} T)  \exp [(LC''_\zeta R (1+\log^{1/4+ \zeta} T)]^{4-2\eta} \,.
\end{aligned} 
\end{equation}
For any $R > r> 10\pi K (t^{min}_{\mathcal B})^{-1}$, by formulas \eqref{eq:osc_n} and~\eqref{eq:upper_bound_strip_bis}, from Lemma~\ref{lemma:valency_bound_onedim}  we derive that there exists a constant $C>0$ such that,  for all 
$t\in (1, 3/2)$, for all fixed $z\in \T$,  and for all sufficiently large $T>0$,  the valency $\nu_\beta(t) \in \N$  of the function of one-complex variable $\beta_{\bf c}(g_{t_n} (a) , \Phi_{g_{t_n} (a),x}(\xi_{y,z}) , T_n)$ is bounded above as follows:
$$
\begin{aligned}
\nu_\beta(t) &\leq  C \log \left ( C'_{\eta, \zeta}  R (1+\log^{3/4+ 3\zeta}  T )  \exp [(LC''_\zeta R (1+\log^{1/4+ \zeta} T)^2 ]^{2-\eta}    \right)  \\
&\leq   C''_{\eta, \zeta} \left(   \log (C'_{\eta,\zeta} R) + \log\log T  +  \log^{1- \frac{\eta}{2} + 2\zeta (2- \eta) } T \right)\,. 
\end{aligned}
$$
By Theorem \ref{thm:bru}, it follows that, for all $\delta \in (1-\eta/2, 1)$,  there exists a  constant $C_{\delta, \zeta}>0$ such that, 
 for all $(x,z, n) \in M\times \Z\times \N$, and for all $\epsilon >0$ we have
\begin{equation}
\label{eq:meas_estimate_n}
\text{Leb}\left( \{ y\in [0,1] \vert    \vert  \beta_{\bf c}(g_{t_n} (a) , \Phi_{g_{t_n} (a),x}(\xi_{y,z}) , T_n) \vert  < \epsilon \} \right) \,< \, 4 \epsilon^{ \frac{1} {C_{\delta,\zeta} \log^{\delta} T }}\,.
\end{equation}
Finally, from the scaling identity~\eqref{eq:scale_n}, and from the measure estimates~\eqref{eq:meas_estimate_n} and Fubini's theorem, it follows that 
for all $\zeta>0$ and for all $\delta \in (1-\eta/2, 1)$ there exists a constant $C_\zeta>0$ such that
$$
\text{\rm vol} (\{x\in M \vert  \vert \beta_{{\bf c}}(a,x,T)\vert \leq  \epsilon \frac{T^{1/2} }{1+ C_\zeta \log^{1/4+\zeta} T} \} \,\leq \,4 \epsilon^{ \frac{1} {C_{\delta, \zeta} \log^{\delta} T }}\,,
$$
as claimed in the statement.

\end{proof}

\begin{corollary}\label{cor:gencase} Let $a=(X,Y,Z_0)$ be as in Lemma \ref{lemma:gencase}. Let $\eta \in (0,1)$ and let ${\bf c} \in \Omega_\infty^{(\eta)}$. For every $\delta \in (1-\eta/2, 1)$ and for 
every $\zeta>0$,  there exist constants $C_{\delta, \zeta}>0$ and $C_\zeta>0$  such that,  for every $\epsilon>0$, we have
\begin{equation}
\text{\rm vol} (\{x \in M \vert  \vert \beta_{\bf c} (a, x, T)\vert \leq \epsilon \frac{T^{1/2}}{C_\zeta \log^{1/4+\zeta} T}  \} ) \leq    4 \epsilon^{\frac{1}{ C_{\delta, \zeta}\log^{\delta} T}}\,.
\end{equation}
\end{corollary}

\section{Correlations. Proof of Theorems \ref{mainthm1} and \ref{mainthm2}}\label{sec:cor}
We will first analyze correlations and then derive Theorems \ref{mainthm1} and \ref{mainthm2} from respectively Corollaries~\ref{cor:compact} and~\ref{cor:gencase}. 
We analyze correlations as follows. Let $\omega_V$ the $V$-invariant volume form. It follows from the definition 
of $V= \alpha X$ that $\omega_V =   \alpha^{-1}  d\text{vol}$.  We have
$$
< h\circ \phi^V_t , g>_{L^2(M,\omega_V)}  = < h\circ \phi^V_t , \frac{g}{\alpha} >_{L^2(M,d\text{vol})}\,,
$$
hence it is equivalent to analyze correlations with respect to Haar volume form $d\text{vol}$. As the Haar volume
form is $Z$-invariant we have
$$
 < h\circ \phi^V_t , g>_{L^2(M,d\text{vol})} = \frac{1}{S}  \int_0^S < h\circ \phi^V_t \circ \phi^Z_s , g \circ \phi^Z_s>_{L^2(M,d\text{vol})} \,ds 
$$
Integrating by parts we finally derive the formula (see the paper \cite{FU}):
$$
\begin{aligned}
\frac{1}{S}  &\int_0^S < h\circ \phi^V_t \circ \phi^Z_s , g \circ \phi^Z_s >_{L^2(M,d\text{vol})} \,ds  \\ &=
\frac{1}{S}   < \int_0^S   h\circ \phi^V_t \circ \phi^Z_s  ds , g \circ \phi^Z_S>_{L^2(M,d\text{vol})}   \\ &- 
\frac{1}{S} \int_0^S  <  \int_0^s  h\circ \phi^V_t\circ \phi^Z_r  dr , Zg \circ \phi^Z_s>_{L^2(M,d\text{vol})} ds \,.
\end{aligned}
$$
We have thus written correlations in terms of integrals
$$
\int_0^S  h\circ \phi^V_t\circ \phi^Z_s  ds\,.
$$
Let $D_t$ denote the function on $M$ defined as
\begin{equation}
\label{eq:time_change_int}
D_t (x) :=    \int_0^t  \frac{Z\alpha}{\alpha} \circ \phi^V_{\tau} d\tau \,.
\end{equation}

We then have the formula
$$
\begin{aligned}
&\int_0^S  h \circ \phi^V_t\circ \phi^Z_s  ds = \int_0^S \frac{1+ D^2_t\circ\phi^Z_s}{1+ D^2_t\circ\phi^Z_s} 
h \circ \phi^V_t\circ \phi^Z_s  ds \\ 
&= \int_0^S \frac{1}{1+ D^2_t\circ\phi^Z_s}  h \circ \phi^V_t\circ \phi^Z_s  ds  + 
 \int_0^S \frac{D^2_t\circ\phi^Z_s}{1+ D^2_t \circ \phi^Z_s}  h \circ \phi^V_t\circ \phi^Z_s  ds\,.
\end{aligned}
$$
We also have
$$
\begin{aligned}
\int_0^S \frac{D^2_t\circ\phi^Z_s}{1+ D^2_t\circ\phi^Z_s} & h \circ \phi^V_t\circ \phi^Z_s  ds =
\int_0^S \frac{D_t\circ\phi^Z_s}{1+ D^2_t\circ\phi^Z_s}   [   (h \circ \phi^V_t \circ \phi^Z_s) (D_t \circ \phi^Z_s) ] ds  \\
 &=  \frac{D_t\circ\phi^Z_S}{1+ D^2_t\circ\phi^Z_S} \int_0^S  (h \circ \phi^V_t \circ \phi^Z_s) (D_t \circ \phi^Z_s)   ds \\
 &-  \int_0^S \frac{d}{ds}[ \frac{D_t\circ\phi^Z_s}{1+ D^2_t\circ\phi^Z_s}]   [\int_0^s  ( h \circ \phi^V_t \circ \phi^Z_r) (D_t \circ \phi^Z_r)  dr] ds   \,\,.
 \end{aligned}
$$
For every $\sigma>0$, let $\gamma^\sigma_{x,t}$ be the path defined as
$$
\gamma^\sigma_{x,t} (s) = (\phi^V_t \circ \phi^Z_s) (x) \,, \quad \text{ for all } s\in [0,\sigma]\,.
$$
We have computed above that
$$
\frac{d\gamma^\sigma_{x,t}}{ds}  (s) = [D_t \circ \phi^Z_s(x)] V(\gamma^\sigma_{x,t}(s)) +   Z (\gamma^\sigma_{x,t}(s))\,.
$$
It follows that 
$$
\int_{\gamma^\sigma_{x,t}} h \hat V   =  \int_0^\sigma   (h \circ \phi^V_t \circ \phi^Z_s)(x)  \, (D_t \circ \phi^Z_s)(x) ds\,.
$$

In other terms, we have the following identity:
$$
\begin{aligned}
\int_0^S \frac{D^2_t\circ\phi^Z_s(x)}{1+ D^2_t\circ\phi^Z_s(x)}  h \circ \phi^V_t\circ \phi^Z_s(x)  ds &=
\frac{D_t\circ\phi^Z_S(x)}{1+ D^2_t\circ\phi^Z_S(x)} [\int_{\gamma^S_{x,t}} h \hat V] \\
\\ & -  \int_0^S \frac{d}{ds}[ \frac{D_t\circ\phi^Z_s(x)}{1+ D^2_t\circ\phi^Z_s(x)}]  [\int_{\gamma^s_{x,t}} h \hat V] ds\,.
 \end{aligned}
$$

It remains to estimate the term
$$
\frac{d}{ds}[ \frac{D_t\circ\phi^Z_s(x)}{1+ D^2_t\circ\phi^Z_s(x)}]  =[ \frac{d}{ds}D_t\circ\phi^Z_s(x)]
[ \frac{1- D^2_t\circ\phi^Z_s(x)}{(1+ D^2_t\circ\phi^Z_s(x))^2}]\,.
$$
Our estimate is thus reduced to a bound on the term
$$
\begin{aligned}
\frac{d}{ds}D_t\circ\phi^Z_s(x)&=   \int_0^t  \{D_\tau  [V(\frac{Z\alpha}{\alpha})\circ \phi^V_{\tau}]
+ Z(\frac{Z\alpha}{\alpha}) \circ \phi^V_{\tau} \}  \circ  \phi^Z_s(x) d\tau \\
&=  [ D_t  (\frac{Z\alpha}{\alpha})\circ \phi^V_{t}] \circ  \phi^Z_s(x) + 
\int_0^t [Z(\frac{Z\alpha}{\alpha})- (\frac{Z\alpha}{\alpha})^2 ] \circ \phi^V_{\tau}\circ \phi^Z_s(x) d\tau \\
& =  [D_t (\frac{Z\alpha}{\alpha})\circ \phi^V_{t}] \circ  \phi^Z_s(x) +
\int_0^t  \alpha  Z (\frac{Z \alpha}{\alpha^2}) \circ \phi^V_{\tau}\circ \phi^Z_s(x) d\tau
\end{aligned}
$$
In particular we conclude that there exists a constant $C_\alpha>0$ such that
$$
\vert \frac{d}{ds}D_t\circ\phi^Z_s(x) \vert \leq  C_\alpha (1+ t^{1/2})\,.
$$

\smallskip
Since the arc $\gamma^\sigma_{x,t}$ is smooth and contained in the weak stable leaf, it follows from 
Lemma~\ref{lemma:Bufetov_funct} and by the H\"older property (or from the scaling property)  that, for all $s>7/2$ there
exists a constant $C_s>0$ such that, for all $t>0$, we have
$$
\vert \int_{\gamma^\sigma_{x,t}} h \hat V \vert  \leq  C_s \vert h \vert_s ( \int_{\gamma^\sigma_{x,t}} \vert \hat V\vert )^{1/2}
\leq  C_s \vert h \vert_s   (1+  \sigma^{1/2} t^{1/4})\,,
$$

For a given $t>1$ and $\epsilon >0$, let now consider the set 
$$
M_{\alpha}(t, \epsilon):= \{ x \in M \vert   \vert D_t (x) \vert  \geq  \epsilon \,t^{1/2} \}\,.
$$
There exists a constant $C_\alpha>0$ such that 
$$
\phi^Z_s  (M_\alpha(t, 2\epsilon))  \subset M_\alpha(t, \epsilon) \,, \quad \text{ for all }s\in (-\frac{\epsilon}{C_\alpha}, 
\frac{\epsilon}{C_\alpha})\,.
$$

\begin{proof}[Proof of Theorems \ref{mainthm1} and \ref{mainthm2}]
By Corollary \ref{cor:compact} for $f=\frac{Z\alpha}{\alpha}\in \Omega$
 there exist constants $C_{a,\alpha}>0$ and  $ \delta_{a,\alpha}>0$ such that we have
$$
\text{vol} \left(M\setminus M_{\alpha}(t, \epsilon)\right) \leq   C_{a,\alpha}  \epsilon^{\delta_{a,\alpha}}\,.
$$
It follows that correlations on the set $M_{\alpha}(t, 2\epsilon)$ can be estimated by the expression
$$
 (1+ t^{3/4})    \Vert   \frac{1}{1+D_t^2} \Vert_{L^1(M_{\alpha}(t, \epsilon))}  \leq  \frac{(1+ t^{3/4})}{1 + \epsilon^2 t} \,,
$$
while the correlation on the complementary set $M\setminus M_\alpha(t, 2\epsilon)$ can be estimated simply as follows:
there exists a constant $C'_{a, \alpha}  >0$ such that 
$$
< h\circ \phi_t^V, g>_{L^2(M\setminus M_\alpha(t, 2\epsilon)} \leq   C'_{a,\alpha}   \epsilon^{\delta_{a,\alpha}} \Vert h \Vert_s (\Vert g \Vert_0 + \Vert Zg\Vert_0)\,.
$$
By optimizing the above estimates we derive a bound for the decay of correlations  of the following form: for every
$h,g\in W^s(M)\cap L^2_0(M)$ and for all $t>0$, we have
$$
< h\circ \phi_t^V, g>_{L^2(M)} \leq   C_{a,\alpha}   (1+t)^{- \frac{\delta_{a,\alpha}}{4(2+ \delta_{a,\alpha})  }} \Vert h \Vert_s (\Vert g \Vert_0 + \Vert Zg\Vert_0)\,.
$$
This finishes the proof of Theorem \ref{mainthm1}. For Theorem \ref{mainthm2}, by an analogous reasoning based on  Corollary~\ref{cor:gencase} we get that for a generic
set of time-changes we have the following bound on correlations. For every $\delta>1/2$, there exists a constant $C_{a,\alpha, \delta}>0$ such that, for all  $h,g\in W^s(M)\cap L^2_0(M)$ and  for all $t\in \R$,  we have
$$
< h\circ \phi_t^V, g>_{L^2(M)} \leq  C_{a,\alpha, \delta} (1+\vert  t \vert) ^{ - \frac{1}{1+ \log^\delta (1+  \vert t \vert )} }  \Vert h \Vert_s (\Vert g \Vert_0 + \Vert Zg\Vert_0)    \,.  
$$
This finishes the proof of Theorem \ref{mainthm2}.
\end{proof}

 \end{document}